\documentclass[12pt]{article}
\usepackage{mystyle}
\usepackage{pgfplots}
\usepackage{caption}
\usepackage{subcaption}
\usepackage[capitalize,noabbrev]{cleveref}
\usepackage{fullpage}
\usetikzlibrary{arrows}

\newcommand{\cosheaf}[1]       {\underline{{#1}}}
\newcommand\define[1]	       {{\bf{#1}}}	

\renewcommand{\vec}            {{\mathrm{Vec}}}
\newcommand{\field}            {{\mathsf{k}}}

\newcommand{\im}               {{\mathrm{im}}\,}
\renewcommand{\ker}            {{\mathrm{ker}}\,}

\newcommand{\Int}               {{\mathrm{Int}}\,}
\newcommand{\Inc}               {{\mathrm{Inc}}}
\newcommand{\Z}                 {{\mathbb{Z}}}
\newcommand{\One}               {{\mathbbm{1}}}

\newcommand{\M}                 {H^\downarrow}

\newcommand{\Ch}                 {C^\downarrow}

\newcommand{\End}               {\mathrm{End}}
\newcommand{\FinAb}             {\mathrm{FinAb}}
\newcommand{\C}                 {\mathbb{C}}
\newcommand{\Gr}                {\mathcal{G}}
\newcommand{\ob}                {\mathrm{ob}\,}

\newcommand{\BD}              {\mathrm{BD}}
\newcommand{\bd}              {\mathrm{bd}}

\usepackage{amstext} % for \text macro
\usepackage{array}   % for \newcolumntype macro
\newcolumntype{L}{>{$}l<{$}}
\newcolumntype{C}{>{$}c<{$}}

\newcommand{\cat}{\mathcal{C}}

\newcommand\blfootnote[1]{%
  \begingroup
  \renewcommand\thefootnote{}\footnote{#1}%
  \addtocounter{footnote}{-1}%
  \endgroup
}

\newcommand{\revision}[1]{#1}
\title{M\"obius Homology}
\author[1]{Amit Patel}
\author[2]{Primoz Skraba}
\affil[1]{Department of Mathematics, Colorado State University}
\affil[2]{School of Mathematical Sciences, Queen Mary University London}
\date{}

\begin{document}

\maketitle

\abstract{\revision{This paper introduces and develops \emph{M\"obius homology}, a homology theory for representations of finite posets into abelian categories. Although the connection between poset topology and M\"obius functions is classical, we go further by establishing a direct connection between poset topology and M\"obius inversions. In particular, we show that M\"obius homology categorifies the M\"obius inversion, as its Euler characteristic coincides with the M\"obius inversion applied to the dimension function of the representation.
We also present a homological version of Rota's Galois Connection Theorem, relating the M\"obius homologies of two posets connected by a Galois connection.} 

\revision{Our main application concerns persistent homology over general posets. We prove that, under a suitable definition, the persistence diagram arises as an Euler characteristic over a poset of intervals, and thus M\"obius homology provides a categorification of the persistence diagram. This furnishes a new invariant for persistent homology over arbitrary finite posets. Finally, leveraging our homological variant of Rota's Galois Connection Theorem, we establish several results about the persistence diagram.}
\blfootnote{This work is partially funded by the Leverhulme Trust grant VP2-2021-008.}
}

% We introduce a homology theory for representations of finite posets into abelian categories we call M\"obius homology.
% The Euler characteristic of the M\"obius homology is the M\"obius inversion.
% In this way, M\"obius homology categorifies the M\"obius inversion.
% As an application, M\"obius homology categorifies the persistence diagram.}

%%%%%%%%%%%%%%%%%%%%%%%%%%%%%%%%%%%%%%%%%%%%%%%
%%%%%%%%%%%%%%%%%%%%%%%%%%%%%%%%%%%%%%%%%%%%%%%
\section{Introduction}
%%%%%%%%%%%%%%%%%%%%%%%%%%%%%%%%%%%%%%%%%%%%%%%
%%%%%%%%%%%%%%%%%%%%%%%%%%%%%%%%%%%%%%%%%%%%%%%

Let $\Gr$ be an abelian group and $P$ a finite poset.
For all functions $m : P \to \Gr$, there exists a unique function $\partial m :  P \to \Gr$ such that
for all $b \in P$, 
	$$m(b) = \sum_{a : a \leq b} \partial m(a).$$
The function $\partial m$ is called the \emph{M\"obius inversion} of $m$.
Often the function $m$ is the shadow of a richer algebraic structure.
That is, there is often an abelian category $\cat$
and a $P$-module, i.e., a functor,
$M  : P \to \cat$ whose \emph{dimension function} is $m$.
Here, the dimension function is the assignment to every $b \in P$
the representative $\big[ M(b) \big]$ of the object $M(b)$ in the Grothendieck group $\Gr$ of $\cat$.
In this case, it is natural to ask for an invariant of $M$ that decategorifies to the M\"obius inversion $\partial m$.

This paper introduces a homology theory for $P$-modules we call \emph{M\"obius homology}.
Every module $M$ gives rise to a simplicial cosheaf over the order complex of $P$.
We define the M\"obius homology of $M$ at an element $b \in P$, denoted $\M_\ast M(b)$, as the homology of the simplicial cosheaf localized to $b$.
Our first result (Theorem~\ref{thm:euler}) states that the Euler characteristic of the M\"obius homology is the M\"obius inversion.
That is, for all $b \in P$,
$$\partial m(b) = \sum_{d  \geq 0} (-1)^d \big[ \M_d M (b) \big].$$

Our second result (Theorem~\ref{thm:rota_homology}) is a homological version of Rota's Galois connection theorem.  Rota's theorem is an important tool for 
studying and computing  M\"obius inversions, and likewise we use our result to prove several properties of M\"obius homology.

% Rota's Galois connection theorem is an important
% tool for computing M\"obius inversions.
% Our second theorem (Theorem~\ref{thm:rota_homology}) is a homological version of Rota's theorem.
% It decategorifies to Rota's theorem and serves as a powerful tool for computing M\"obius homology. 

\paragraph{Previous Work} 
Poset topology is the study of the order complex associated with a poset. 
This field originates from Rota's seminal 1964 paper on M\"obius functions~\cite{gian1964foundations} 
and has since evolved into a rich area of study motivated by questions from diverse disciplines~\cite{wachs2006poset}. 
The M\"obius function is of particular interest due to its connection with the Euler characteristic of the order complex~\cite{rota_euler}. 
In this work, we explore the topology of $P$-modules and their relationship to M\"obius inversion.

The study of $P$-modules as sheaves dates back to at least Deheuvels~\cite{Deheuvels1962}. 
Building on this perspective, Baclawski introduced Whitney homology for modules over geometric lattices, 
providing a categorification of Whitney numbers~\cite{BACLAWSKI1975125}. 
In a subsequent paper, Baclawski extended this approach by categorifying Rota's Galois connection theorem 
as a Leray spectral sequence~\cite{baclawski1977galois}.

Recent work has advanced the categorification of M\"obius inversion for posets arising from stratifications. 
Petersen~\cite{petersen2017spectral} constructed a spectral sequence for stratified spaces 
that provides a categorical lift of M\"obius inversion. This construction computes the compactly supported cohomology of an open stratum 
in terms of the cohomology of closed strata. 
More recently, Ayala, Mazel-Gee, and Rozenblyum categorified M\"obius inversion in the context of poset modules valued in a stable $\infty$-category~\cite{ayala2024stratified}. 
Their approach generalizes Petersen's results. 
Both works appear to extend the earlier contributions of Yanagawa in the setting of cellular spaces~\cite{yanagawa2005dualizing}. 
This broader line of work, including the results presented in this paper, suggests a deeper connection with Euler calculus~\cite{viro1988euler, SCHAPIRA199183}.

Unlike previous approaches, our method avoids spectral sequences and instead provides a homology theory that directly categorifies M\"obius inversion.

\paragraph{Application: Persistent Homology} 
Our discovery of M\"obius homology is motivated by questions in persistent homology.
Traditionally, persistent homology starts with a family of topological spaces parameterized by a finite, totally ordered poset $P$.
Apply homology using field coefficients and 
the result is a $P$-module of vector spaces 
$M_\ast : P \to \vec$, one for each dimension.
The \emph{persistence diagram} of each $M_\ast$
is a complete combinatorial invariant that encodes the birth and death of its generators along $P$.

There are many equivalent definitions of the persistence diagram.
The definition proposed by Cohen-Steiner, Edelsbrunner, and Harer uses an inclusion-exclusion formula~\cite{stability-persistence}.
From this viewpoint, the persistence diagram of $M_\ast$ is the M\"obius inversion of a $\Z$-valued function on the poset of intervals of $P$~\cite{Patel2018}.
There are now two well developed generalizations of this definition to modules over any finite poset.
The persistence diagram of Kim and M\'emoli is a M\"obius inversion of the rank invariant of the module~\cite{kim2018generalized} whereas the persistence diagram of G\"ulem, McCleary, and Patel is the M\"obius inversion of the birth-death invariant of the module~\cite{McCPa20-edit-distance, GulenMcCleary}.
When $P$ is not totally ordered, neither are complete invariants of the module.
We show how M\"obius homology categorifies the birth-death approach and is a stronger invariant than simply the M\"obius inversion.
Finally, we demonstrate how the homological version of Rota's Galois connection theorem helps in computing persistent M\"obius homology.

\section{Background}
%%%%%%%%%%%%%%%%%%%%%%%%%%%%%%%%%%%%%%%%%%%%%%%
%%%%%%%%%%%%%%%%%%%%%%%%%%%%%%%%%%%%%%%%%%%%%%%

This section provides a brief introduction to M\"obius inversions, simplicial cosheaves, the Grothendieck group, and Galois connections. 

%There are many options for a more detailed introduction to the M\"obius inversion; see, for example~\cite{stanley_2011}.
%Unfortunately, the literature on simplicial cosheaves is sparse.
%A simplicial cosheaf is a special case of a cellular cosheaf and both are examples of constructible cosheaves~\cite{curry_patel}. 
%See~\cite{shepard} for an introduction to cellular sheaves, which are dual to cellular cosheaves~\cite{CURRY2018966}.

%%%%%%%%%%%%%%%%%%%%%%%%%%%%%%%%%%%%%%%%%%%%%%%
\subsection{M\"obius Inversion}
%%%%%%%%%%%%%%%%%%%%%%%%%%%%%%%%%%%%%%%%%%%%%%%
We begin with the M\"obius inversion formula.
For a more thorough introduction to M\"obius inversion, see~\cite{stanley_2011}.

Let $P$ be a finite poset. For $a \leq c$, the (closed) \emph{interval} $[a,c]$ is the set 
\[
[a,c] := \{ b \in P : a \leq b \leq c \}.
\]
Let $\mathcal{I}(P)$ denote the set of all intervals in $P$. The \emph{$\Z$-incidence algebra} on $P$, denoted $\Inc(P)$, is the set of all functions $\alpha : \mathcal{I}(P) \to \Z$ with three operations: scaling, addition, and multiplication.
The first two are clear.
Multiplication of two functions $\alpha$ and  $\beta$ is
$$(\alpha \ast \beta)[a,c] := \sum_{b:a \leq b \leq c} \alpha[a,b] \cdot \beta[b,c].$$
The multiplicative identity is 
    \begin{equation*}
    \One [a,b] = \begin{cases}
    1 & \text{if $a = b$}  \\
    0 & \text{otherwise}.
    \end{cases}
    \end{equation*}
The \emph{zeta function} is 
$\zeta [a,b] = 1$, for all $a \leq b$.
The zeta function is invertible, and its inverse, denoted $\mu$, is the \emph{M\"obius function}.

\begin{defn}
Let $P$ be a finite poset and $\Gr$ an abelian group.
The \define{M\"obius inversion} of any function $f : P \to \Gr$ is the function
$\partial f : P \to \Gr$
defined as
\begin{equation}\label{eq:inverion_formula}
\partial f(c) = \sum_{b: b \leq c} f(b) \cdot \mu[b,c].
\end{equation}
\end{defn}
The M\"obius inversion is the unique function satisfying the following equation for all~$c \in P$:
\begin{align*}
\sum_{b: b \leq c} \partial f(b) &= \sum_{b: b \leq c} \partial f(b) \cdot \zeta[b,c]  \\
&=
\sum_{b: b \leq c} \left ( \sum_{a: a \leq b} f(a) \cdot \mu[a,b] \right ) \cdot \zeta[b,c] \\
&= \sum_{a: a \leq c} f(a) \left( \sum_{b : a \leq b  \leq c} \mu[a,b] \cdot \zeta [b,c] \right) \\
& =
\revision{\sum_{a: a \leq c} f(a) \cdot \One[a,c] } \\ 
&= f(c).
\end{align*}
That is, $\partial f$ is the \revision{combinatorial} derivative of $f$.

%%%%%%%%%%%%%%%%%%%%%%%%%%%%%%%%%%%%%%%%%%%%%%%%%%%%%%%%%%%%
\subsection{Simplicial Cosheaves}
%%%%%%%%%%%%%%%%%%%%%%%%%%%%%%%%%%%%%%%%%%%%%%%%%%%%%%%%%%%%

We introduce simplicial cosheaves as functors from a simplicial complex to an abelian category. The homology of these cosheaves and their Euler characteristic, which lives in the Grothendieck group of the abelian category, will be essential to our main theorems. For more background on cellular (co)sheaves, see~\cite{ghrist2014elementary}, and for the Grothendieck group, see~\cite{weibel2013kbook}.

Let $K$ be a finite simplicial complex.
We write $\tau \geq \sigma$ to mean $\tau$ is a coface of $\sigma$ and 
$\tau >_1 \sigma$ to mean $\dim \tau = \dim(\sigma) +1$.
We also think of $K$ as a category where $\tau \to \sigma$ whenever $\tau \geq \sigma$.
For the moment, let $\cat$ be any abelian category.

\begin{defn}
A \define{simplicial cosheaf} over $K$ 
valued in $\cat$ is a functor $\cosheaf{A} : K \to \cat$.
\end{defn}

We now construct the chain complex $C_\bullet (K; \cosheaf{A})$ associated to a simplicial cosheaf~$\cosheaf{A}$.
Orient every simplex of $K$.
The $d$-th chain object is
$$C_d (K; \cosheaf{A}) := \bigoplus_{\sigma: \dim(\sigma) = d} \cosheaf{A}(\sigma).$$
For $\tau >_1 \sigma$, restrict the orientation on $\tau$ to an orientation on $\sigma$. 
Write $[\tau:\sigma] = 1$ if the
restriction agrees with the orientation on 
$\sigma$, and write $[\tau:\sigma] = -1$ if the restriction disagrees.
The boundary operator $\partial_d : C_d (K; \cosheaf{A}) \to C_{d-1} (K; \cosheaf{A})$ is 
generated summand-wise as follows.
For a  $d$-simplex $\tau$, let 
$\iota_\tau : \cosheaf{A}(\tau) \to C_d(K; \cosheaf{A})$ be the canonical morphism into the direct sum, and let $\pi_\tau : C_d(K;\cosheaf{A}) \to \cosheaf{A}(\tau)$ be the canonical morphism out of the direct sum.
Define the restriction of $\partial_d$ to $\tau$ as
$$\partial_d  |_\tau := \sum_{\sigma: \tau >_1 \sigma} [\tau:\sigma] \cdot \big(
\iota_\sigma \circ \cosheaf{A}(\tau \geq \sigma)  \circ \pi_\tau  \big).$$
The full boundary operator is the sum 
$\partial_d := \sum_{\tau : \dim \tau = d} \partial_d |_\tau.$
Of course, the ordinary boundary operator applied twice to any simplex is zero.
This fact combined with commutativity of $\cosheaf{A}$, implies
$\partial_{d-1} \circ \partial_d = 0$.
Thus, we have a chain complex
    \begin{equation*}
    \begin{tikzcd}
    C_\bullet (K; \cosheaf{A}) : \cdots \ar[r]  & C_2(K; \cosheaf{A}) \ar[r, "\partial_2"] & C_1 (K;\cosheaf{A})
     \ar[r, "\partial_1"]  & C_0 (K; \cosheaf{A} ) \ar[r,  "\partial_0"]  &  0.
    \end{tikzcd}
    \end{equation*}

\begin{defn}
The \define{$d$-th cosheaf homology} of $\cosheaf{A}$, denoted $H_d ( K; \cosheaf{A})$, is the $d$-th homology object of the chain complex $C_\bullet (K; \cosheaf{A})$.
\end{defn}

We are also interested in relative simplicial cosheaf homology.
Let $L \subseteq K$ be a subcomplex,
and let $\cosheaf{B} : L \to \cat$ be the restriction of the cosheaf $\cosheaf{A}$ to $L$.
The associated relative chain complex is
    \begin{equation*}
    \begin{tikzcd}
    C_\bullet (K, L; \cosheaf{A}) : \cdots \ar[r]  & \dfrac{C_2(K; \cosheaf{A})}{C_2(L; \cosheaf{B})} \ar[r, "\partial_2"] & \dfrac{C_1 (K;\cosheaf{A})}{C_1 (L;\cosheaf{B})}
     \ar[r, "\partial_1"]  & \dfrac{C_0 (K; \cosheaf{A} )}{C_0 (L; \cosheaf{B} )} \ar[r,  "\partial_0"]  &  0.
    \end{tikzcd}
    \end{equation*}

\begin{defn}
The \define{$d$-th relative cosheaf homology} of $\cosheaf{A}$ with respect to the subcomplex $L \subseteq K$, denoted $H_d ( K, L; \cosheaf{A})$, is the $d$-th homology object of 
 the chain complex $C_\bullet (K, L; \cosheaf{A})$.
\end{defn}

We now explore the Euler characteristic
of a simplicial cosheaf
valued in an abelian category~$\cat$.
Assume $\cat$ is equivalent to a small category.
Denote by $[A]$ the isomomorphism class of an object $A$ in $\cat$.

\begin{defn}
The \define{Grothendieck group} of $\cat$, denoted $K(\cat)$, is the free abelian group generated by the set of isomorphism classes of objects in $\cat$ with the relation
$[B] = [A] + [C]$ for every short exact sequence
$0 \to A \to B \to C \to 0$ in $\cat$.
\end{defn}

We now give three working examples of abelian
categories along with their Grothendieck groups.
The calculation of each group
is left to the reader.

\begin{ex}
If $\vec$ is the category of 
finite dimensional $\field$-vector spaces, 
for some fixed field~$\field$, then
$K(\vec) \cong \Z$.
For an object $A$, the corresponding element $[A] \in K(\vec)$ is the dimension of $A$.
\end{ex}

\begin{ex}
Let $\End(\C)$ be the category of 
 endomorphisms
on finite dimensional complex vector spaces.
That is, an object
is an endomorphism $\phi : \C^m \to \C^m$, and 
a morphism from an object $\phi : \C^m \to \C^m$ to an object $\psi : \C^n \to \C^n$ is a linear map $f : \C^m \to \C^n$ such that $f \circ \phi = \psi \circ f$.
Its Grothendieck group, $K\big( \End (\C) \big)$,
is isomorphic to $\oplus_{\lambda \in \C} \Z$.
For an object~$A$, the corresponding element $[A] \in K\big( \End (\C) \big)$ is the map $\C \to \Z$ 
that assigns to every $\lambda \in \C$ the number of times $\lambda$ occurs as an eigenvalue of $A$.
\end{ex}

\begin{ex}
Let $\FinAb$ be the category of finite
abelian groups.
Its Grothendieck group, $K(\FinAb)$, is isomorphic to $\oplus_{p \text{ prime}} \Z$.
For an object~$A$, the corresponding element $[A] \in K(\FinAb)$
is the assignment to each prime the number of times it occurs in the prime decomposition
of $A$.
For example, if $A$ is $\Z / 16 \Z$, then $[A]$ is the assignment $2 \mapsto 4$.
\end{ex}

\begin{defn} \label{defn:total_euler}
The \define{Euler characteristic} of a simplicial cosheaf $\cosheaf{A} : K \to \cat$, denoted $\chi(K; \cosheaf{A})$, is the following element of $K(\cat)$:
\begin{equation}\label{eq:euler1}
\chi (K;  \cosheaf{A}) :=  \sum_{d \geq 0}  \sum\limits_{\sigma : \dim(\sigma) = d} (-1)^d \big[ \cosheaf{A}(\sigma) \big].
\end{equation}
\end{defn}

If we think of $\big[ H_i (K; \cosheaf{A}) \big] \in K(\cat)$ as the $i$-th ``Betti number'' of $\cosheaf{A}$, then the Euler characteristic of $\cosheaf{A}$ 
is the familiar signed sum of Betti numbers as follows.

\begin{prop} \label{prop:cosheaf_euler}
$\chi (K; \cosheaf{A})  = \sum_{d \geq 0} (-1)^d \big[H_d(K;\cosheaf{A}) \big].$
\end{prop}
\begin{proof}
Recall the chain complex $C_\bullet(K; \cosheaf{A})$.
Let $Z_d (K; \cosheaf{A}) := \ker \partial_d$
and $B_d (K; \cosheaf{A}) := \im \partial_{d+1}$.
Consider the following short exact sequences:
    \begin{equation*}
    \begin{tikzcd}
    0 \ar[r] & Z_d (K;  \cosheaf{A}) \ar[r, hookrightarrow] & C_d (K; \cosheaf{A}) \ar[r, twoheadrightarrow, "\partial_d"] &  
        B_{d-1} (K;  \cosheaf{A}) \ar[r] & 0 \\
    0 \ar[r] & B_d (K  ; \cosheaf{A}) \ar[r, hookrightarrow] & Z_d (K ; \cosheaf{A}) \ar[r, twoheadrightarrow] &  
        H_{d} (K ;\cosheaf{A}) \ar[r] & 0.
    \end{tikzcd}
    \end{equation*}
The two sequences imply $\big[ C_d (K; \cosheaf{A}) \big] = \big[ Z_d (K ;\cosheaf{A})\big] + \big[B_{d-1} (K ;\cosheaf{A})\big]$ and
$\big[ H_d (K ; \cosheaf{A} ) \big] = 
\big[ Z_d (K ; \cosheaf{A}) \big]  - 
\big[ B_{d} (K ;  \cosheaf{A}) \big]$
in $K(\cat)$. Thus
    \begin{align*}
    \chi(K; \cosheaf{A})  
    &= \sum_{d \geq 0} \sum_{\sigma: \dim(\sigma) = d} (-1)^d \big[ \cosheaf{A}(\sigma) \big] \\
    &= \sum_{d \geq 0} (-1)^d \big[ C_d (K ; \cosheaf{A})\big] \\
    &= \sum_{d \geq 0} (-1)^d \big[ Z_d (K  ;\cosheaf{A})\big] + \sum_{d \geq 0} (-1)^{d} \big[ B_{d-1} (K  ; \cosheaf{A})\big] \\
    &= \sum_{d \geq 0} (-1)^d \left( \big[ Z_d (K ; \cosheaf{A})\big] - \big[ B_{d} (K  ; \cosheaf{A})\big] \right) \\
   &= \sum_{d \geq 0} (-1)^d \big[ H_d (K;  \cosheaf{A})\big].
    \end{align*}
\end{proof}

\begin{defn}\label{defn:reuler}
The \define{relative Euler characteristic} of $\cosheaf{A}$ with respect to the subcomplex $L \subseteq K$, denoted $\chi(K,L; \cosheaf{A})$, is
\begin{equation*}
    \chi(K,L; \cosheaf{A}) := \sum_{d \geq 0} \sum_{\substack{\sigma \in K \setminus L \\ \dim(\sigma) = d}}
    (-1)^d \big[ \cosheaf{A}(\sigma) \big].
\end{equation*}
\end{defn}

Similarly, the relative Euler characteristic is the signed sum of relative Betti numbers.

\begin{prop}\label{prop:relative_reuler}
$\chi(K, L; \cosheaf{A}) = \sum_{d \geq 0} (-1)^d \big[ H_d (K, L; \cosheaf{A}) \big].$
\end{prop}

%%%%%%%%%%%%%%%%%%%%%%%%%%%%%%%%%%%%%%%%%%%%%%%
\subsection{Galois Connections}
%%%%%%%%%%%%%%%%%%%%%%%%%%%%%%%%%%%%%%%%%%%%%%%

We now review some fundamental properties of Galois connections. A Galois connection between posets $P$ and $Q$ can be understood as a pair of adjoint functors.  
It is well known that such adjoint functors induce a homotopy equivalence between the order complexes $\Delta P$ and $\Delta Q$; see Quillen~\cite[\S 1.3]{QUILLEN1978101}.  
In this section, we construct an explicit chain homotopy equivalence between $\Delta P$ and $\Delta Q$, which will be a key component in the proof of Theorem~\ref{thm:rota_homology}.

\begin{defn}
A \define{(monotone) Galois connection} from a poset $P$ to a poset $Q$, written
$f : P \leftrightarrows Q : g$, are monotone functions $f : P \to Q$ and $g : Q \to P$
such that for all $a \in P$ and for all $x \in Q$, $f(a) \leq x$ if and only if $a \leq g(x).$
The function $f$ is called the \emph{left adjoint} of $g$ and $g$ is called the \emph{right adjoint} of $f$
\end{defn}

\begin{ex}
Consider the following Galois connections
$f : P \leftrightarrows Q : g$ where $f$ 
is given by the red dashed arrows and $g$ by the blue dashed arrows:

\begin{tikzpicture}
% Poset P
\node (0) at (-0.5,2.1) {$\textcolor{red}{P}$};
\node (a) at (0,0) {};
\node (b) at (2,0) {};
\node (c) at (0,2) {};
\node (d) at (2,2) {};
\filldraw[red] (a) circle (2pt);
\filldraw[red] (b) circle (2pt);
\filldraw[red] (c) circle (2pt);
\filldraw[red] (d) circle (2pt);
\draw[->] (a) -- (b);
\draw[->] (a) -- (c);
\draw[->] (b) -- (d);
\draw[->] (c) -- (d);

% Poset Q
\node (1) at (3.9,2.1) {$\textcolor{blue}{Q}$};

\node (e) at (0,-1.5) {};
\node (f) at (3.5,-1.5) {};
\node (g) at (3.5,2) {};
\filldraw[blue] (e) circle (2pt);
\filldraw[blue] (f) circle (2pt);
\filldraw[blue] (g) circle (2pt);
\draw[->] (e) -- (f);
\draw[->] (f) -- (g);
% Galois connection

\draw[->,red,dashed] (a) to[out=240,in=120] (e);
\draw[->,blue,dashed] (e) to[out=60,in=300] (a);
\draw[->,blue,dashed] (g) to[out=150,in=30] (d);
\draw[->,red,dashed] (d) to[out=-30,in=210] (g);
\draw[->,blue,dashed] (g) to[out=150,in=30] (d);
\draw[->,red,dashed] (b) to[out=-90,in=170] (f);
\draw[->,blue,dashed] (f) to[out=100,in=-0] (b);
\draw[->,red,dashed] (c) to[out=40,in=140] (g);
\end{tikzpicture}
%\hspace{1cm}
\begin{tikzpicture}
\node (0) at (0,0) {};
\node (1) at (8.4,0.5) {$\textcolor{blue}{Q}$};
\node (2) at (8.4,3) {$\textcolor{red}{P}$};

\foreach \x in {0,...,4}{
 \node  (x\x) at (2*\x,0.5) {};
 \filldraw[blue] (x\x) circle (2pt);
}
\foreach \x in {0,...,4}{
 \node  (y\x) at (2*\x,3) {};
 \filldraw[red] (y\x) circle (2pt);
}
  \foreach \x/\y in {0/1,1/2,2/3,3/4} {
  \draw[->] (x\x) -- (x\y);
  \draw[->] (y\x) -- (y\y);  
    }
\draw[->,blue,dashed] (x0) to[out=60,in=300] (y0);
\draw[->,red,dashed] (y0) to[out=240,in=120] (x0);
\draw[->,blue,dashed] (x4) to[out=60,in=300] (y4);
\draw[->,red,dashed] (y4) to[out=240,in=120] (x4);
\draw[->,red,dashed] (y1) to[out=-15,in=80] (x2);
\draw[->,blue,dashed] (x2) to[out=170,in=-90] (y1);
\draw[->,blue,dashed] (x1) to[out=90,in=-30] (y0);
\draw[->,blue,dashed] (x3) to[out=60,in=300] (y3);
\draw[->,red,dashed] (y3) to[out=240,in=120] (x3);
\draw[->,red,dashed] (y2) to[out=-80,in=170] (x3);
\end{tikzpicture}
\end{ex}

\begin{lem} \label{lem:galois}
Every Galois connection $f : P \leftrightarrows Q : g$ satisfies the following properties:
    \begin{enumerate}[label=(\roman*)]
        \item For all $a \in P$, $a \leq g \circ f(a)$.
        \item For all $x \in Q$, $f \circ g (x) \leq x$.
        \item If $f(a) = x$, then $f \circ g(x) = x$.
        \item $f$ is surjective if and only if $g$ is injective.
        \item $f$ is injective if and only if $g$ is surjective.
    \end{enumerate}
\end{lem}
\begin{proof} 
\begin{enumerate}[label=(\roman*)]
    \item In the axiom of a Galois connection above, let $x = f(a)$.
    Then the inequality $f(a) \leq f(a)$ implies $a \leq g \circ f(a)$.
    \item In the axiom of a Galois 
    connection above, let $a = g(x)$.
    Then the inequality $g(x) \leq g(x)$ implies $f \circ g (x) \leq x$.
    \item The inequality $f(a) \leq x$ implies $a \leq g(x)$. Since $f$ is monotone, $x = f(a) \leq f \circ g(x)$. 
    Combined with (ii), $f \circ g(x) = x$.
    \item Assume $f$ is surjective. 
    Fix distinct elements $x, y \in Q$, and suppose, for the sake of contradiction,  
    $g(x) = g(y)$.
    Since $f$ is surjective, $f(a) = x$ and $f(b) = y$ for some $a,b \in P$. 
    By (iii), $f \circ g(x) = x$ and $f \circ g(y) = y$. 
    Thus $x = y$ contradicting $x \neq y$.
    
    Assume $f$ is not surjective. 
    Then there is an $x \in Q$ such that $f^{-1}(x) = \emptyset$. 
    Let $a = g(x)$ and let $y = f (a)$. 
    By (iii), $x \neq y$. 
    The inequality $f(a) \leq y$ implies $a \leq g(y)$. 
    By (ii), $f \circ g(x) = y \leq x$.
    Apply $g$ to get $g(y) \leq g(x) = a$.
    Thus $g$ is not injective because $g(x) = g(y) =a$.

    \item Assume $f$ is injective.
    If $a \in P$ and $x = f(a)$, then, by (iii),
    $f \circ g(x) = x$.
    Thus $g(x) = a$ making $g$ surjective.

    Assume $f$ is not injective.
    Then there are are distinct elements 
    $a, b \in P$ such that $f(a) = f(b) = x$.
    By (iii), $g(x)$ is an element of $f^{-1}(x)$, which contains both $a$ and $b$.
    Thus $g$ cannot be surjective.
\end{enumerate}
\end{proof}

Fix a Galois connection $f : P \leftrightarrows Q : g$.
The function $f$ takes a chain $a_0 < \cdots < a_n$ in $P$ to the chain $f(a_0) \leq \cdots \leq f(a_n)$ in $Q$. 
Thus,~$f$ gives rise to a simplicial map $\Delta f : \Delta P \to \Delta Q$.
Similarly, $g$ gives rise to a simplicial map $\Delta g : \Delta Q \to \Delta P$.
Denote by~$C_\bullet (\Delta P; \Z)$ and~$C_\bullet ( \Delta Q; \Z)$ the 
standard simplicial chain complexes using
$\Z$-coefficients
associated to $\Delta P$ and $\Delta Q$, respectively.
The simplicial maps $\Delta f$ and $\Delta g$ induce the chain maps
    \begin{align*}
    \Delta f_\bullet : C_\bullet (\Delta P; \Z) \to C_\bullet (\Delta Q; \Z) &&
    \Delta g_\bullet  : C_\bullet (\Delta Q; \Z) \to C_\bullet (\Delta P; \Z).
    \end{align*}

\begin{prop}\label{prop:chain_eq}
The compositions $\Delta g \circ \Delta f$ and $\Delta f \circ \Delta g$ are chain homotopic to the identity on $C_\bullet(\Delta P; \Z)$ and
$C_\bullet(\Delta Q; \Z)$, respectively.
In other words, the two chain complexes
are chain homotopy equivalent and therefore their homologies are isomorphic.
\end{prop}
\begin{proof}
    See Appendix~\ref{sec:appendix_one}.
\end{proof}

%%%%%%%%%%%%%%%%%%%%%%%%%%%%%%%%%%%%%%%%%%%%%%%
%%%%%%%%%%%%%%%%%%%%%%%%%%%%%%%%%%%%%%%%%%%%%%%
\section{M\"obius Inversions and Euler Characteristics}
%%%%%%%%%%%%%%%%%%%%%%%%%%%%%%%%%%%%%%%%%%%%%%%
%%%%%%%%%%%%%%%%%%%%%%%%%%%%%%%%%%%%%%%%%%%%%%%

We define two invariants associated to a $P$-module valued in an abelian category.
The first is the M\"obius inversion of the module's \emph{dimension function}.
The second is \emph{M\"obius homology}.
Our main theorem (Theorem~\ref{thm:euler}) shows that the M\"obius inversion is the Euler characteristic of the M\"obius homology.

\begin{defn}
Let $P$ be a finite poset and $\cat$ any small abelian category.
A \define{$P$-module} is a functor $M  : P \to \cat$.
\end{defn}

%%%%%%%%%%%%%%%%%%%%%%%%%%%%%%%%%%%%%%%%%%%%%%%
\subsection{Dimension Function}
%%%%%%%%%%%%%%%%%%%%%%%%%%%%%%%%%%%%%%%%%%%%%%%

The first invariant considers only the objects of $M$ while forgetting all its morphisms.

\begin{defn}\label{defn:dimension}
The \define{dimension function} of $M$ is the function $m  : P \to K(\cat)$
that assigns to every $a \in P$ the element $\big[ M(a) \big] \in K(\cat)$.
\end{defn}

We now look at two similar examples of the dimension function~$m$ and its M\"obius inversion~$\partial m$ in two different categories.

\begin{ex}
\label{ex:first_mobius}
Consider the poset $P$ and the two $P$-modules $M$ and $N$ in Figure~\ref{fig:first_example}.
The dimension function $m : P \to \Z$ for both modules are the same.
The M\"obius inversion $\partial m$ is $\partial m(b) = 1$,
$\partial m(a) = 0$, and $\partial m(c) = 1$.
   \begin{figure}
	\begin{equation*}
	\begin{tikzcd}[ampersand replacement=\&]
	M \& \field \ar[rr, "\footnotesize{\begin{pmatrix}1 \\ 0 \end{pmatrix}}"] \&\& \field^2 \&\& \field 
	\ar[ll, "\footnotesize{\begin{pmatrix}1 \\ 0 \end{pmatrix}}"'] \\
	N \& \field \ar[rr, "\footnotesize{\begin{pmatrix}1 \\ 0 \end{pmatrix}}"] \&\& \field^2 \&\& \field 
	\ar[ll, "\footnotesize{\begin{pmatrix}  0 \\ 1 \end{pmatrix}}"'] \\
	P \& b \ar[rr, "\leq"] \&\& a \&\& c \ar[ll,  "\geq"']
	\end{tikzcd}
	\end{equation*}
	\caption{Poset $P$ along with two $P$-modules $M$ and $N$ valued in $\vec$.}
	\label{fig:first_example}
	\end{figure}
\end{ex}

\begin{ex}
\label{ex:second_mobius}
Consider the poset $P$ and the two $P$-modules $M$ and $N$ in Figure~\ref{fig:second_example}.
Recall an element of $K\big(\mathrm{End}(\mathbb{C}) \big)$ is the assignment to every complex number a multiplicity.
The dimension function $m : P \to K(\cat)$ for both modules are the same.
That is, $m(b) = \{ \lambda \mapsto 1 \}$, $m(a) = \{ \lambda \mapsto 2 \}$,
and $m(c) = \{ \lambda \mapsto 2\}$.
The M\"obius inversion $\partial m$ is $\partial m(b) = \{ \lambda \mapsto 1 \}$,
$\partial m(a) = 0$, and $\partial m(c) = \{ \lambda \mapsto 1\}$.

\begin{figure}
\begin{equation*}
\begin{tikzcd}[ampersand replacement=\&]
	M \& \C \arrow["\footnotesize{\begin{pmatrix} \lambda \end{pmatrix}}"', loop, distance=1.5em, in=110,out=70,start anchor={[xshift=1.4ex,yshift=-0.5ex]north}, end anchor={[xshift=-1.2ex,yshift=-0.5ex]north}]
  \ar[rr, "\footnotesize{\begin{pmatrix}1 \\ 0 \end{pmatrix}}"] \&\& \C^2 
  \arrow["\footnotesize{\begin{pmatrix} \lambda & 0 \\ 0 & \lambda\end{pmatrix}}"', loop, distance=1.5em, in=110,out=70,start anchor={[xshift=1.4ex,yshift=-0.50ex]north}, end anchor={[xshift=-1.2ex,yshift=-0.5ex]north}]
   \&\& \C 
        \arrow["\footnotesize{\begin{pmatrix} \lambda \end{pmatrix}}"', loop, distance=1.5em, in=110,out=70,start anchor={[xshift=1.4ex,yshift=-0.5ex]north}, end anchor={[xshift=-1.2ex,yshift=-0.5ex]north}]
	\ar[ll, "\footnotesize{\begin{pmatrix}1 \\ 0 \end{pmatrix}}"'] \\  
    \\
    \\
	N \& \C \arrow["\footnotesize{\begin{pmatrix} \lambda \end{pmatrix}}"', loop, distance=1.5em, in=110,out=70,start anchor={[xshift=1.4ex,yshift=-0.50ex]north}, end anchor={[xshift=-1.2ex,yshift=-0.5ex]north}]
    \ar[rr, "\footnotesize{\begin{pmatrix}1 \\ 0 \end{pmatrix}}"] \&\& \C^2 
    \arrow["\footnotesize{\begin{pmatrix} \lambda & 0 \\ 0 & \lambda\end{pmatrix}}"', loop, distance=1.5em, in=110,out=70,start anchor={[xshift=1.4ex,yshift=-0.50ex]north}, end anchor={[xshift=-1.2ex,yshift=-0.5ex]north}]
    \&\& \C
      \arrow["\footnotesize{\begin{pmatrix} \lambda \end{pmatrix}}"', loop, distance=1.5em, in=110,out=70,start anchor={[xshift=1.4ex,yshift=-0.50ex]north}, end anchor={[xshift=-1.2ex,yshift=-0.5ex]north}]
	\ar[ll, "\footnotesize{\begin{pmatrix}  0 \\ 1 \end{pmatrix}}"'] \\
	P \& b \ar[rr, "\leq"] \&\& a \&\& c \ar[ll,  "\geq"']
	\end{tikzcd}
	\end{equation*}
	\caption{Poset $P$ along with two $P$-modules $M$ and $N$ valued in $\End(\C)$.}
\label{fig:second_example}
\end{figure}
\end{ex}

%%%%%%%%%%%%%%%%%%%%%%%%%%%%%%%%%%%%%%%%%%%%%%%
\subsection{Order cosheaf}
%%%%%%%%%%%%%%%%%%%%%%%%%%%%%%%%%%%%%%%%%%%%%%%
For our second invariant, we start by building a simplicial complex.
A chain of \emph{length}~$n$ in~$P$ is a sequence $a_0  < \cdots < a_n$ of $n+1$ distinct elements.
A \emph{subchain} of a chain $a_0 < \cdots < a_n$ is a chain obtained by deleting any number of its elements.

\begin{defn}
The \define{order complex} of $P$, 
denoted $\Delta P$,
is the simplicial complex whose (open) $i$-simplices are chains of length~$i$.
A simplex $a_0 < \cdots < a_j$ is a coface of a simplex $b_0 < \cdots < b_i$
if the later is a subchain of the former.
\end{defn}

Let $\min, \max : \Delta P \to P$ be the functions that assign to each simplex its minimal and maximal elements, respectively, in its corresponding chain. Note that if $\tau \geq \sigma$ in $\Delta P$, then $\min(\tau) \leq \min(\sigma)$ and $\max(\sigma) \leq \max(\tau)$ in $P$.

\begin{defn}
The \define{order cosheaf} of $M$ is the simplicial cosheaf $\cosheaf{M}: \Delta P \to \cat$ defined as the following composition:
\begin{equation*}
    \begin{tikzcd}
        \Delta P \ar[r, "\min"] & P \ar[r, "M"] & \cat
    \end{tikzcd}
\end{equation*}
\end{defn}

We now consider a local order cosheaf homology.

\begin{defn}
The \define{lower complex} of $b \in P$ is the subcomplex 
$\Delta P_{\leq b} := \{ \sigma \in \Delta P : \max(\sigma) \leq b \}$ of $\Delta P$.
The \define{strict lower complex} of $b \in P$ is the subcomplex
$\Delta P_{< b} := \big  \{ \sigma \in \Delta P : \max( \sigma) < b \big  \}$ of $\Delta P$.
\end{defn}

\begin{defn}
The \define{M\"obius chain complex module} of $M :P \to \cat$, denoted $\Ch_\bullet M : P \to \mathrm{Ch}(\cat)$, is the assignment to every $a \in P$ the relative chain complex $C_\bullet( \Delta P_{\leq a}, \Delta P_{< a}; \cosheaf{M} )$ and to every $a \leq b$
the morphism
$$C_\bullet \big( \Delta P_{\leq a}, \Delta P_{< a}; \cosheaf{M} \big) \rightarrow C_\bullet \big(\Delta P_{\leq b}, \Delta P_{< b}; \cosheaf{M}  \big)$$
induced by the inclusion of the pair of complexes.
Apply the homology functor and the result is the \define{M\"obius homology module} of $M$, denoted $\M_\ast M : P \to \cat$.
\end{defn}

The internal morphisms of the M\"obius homology module $\M_\ast M$ are all zero.
This is because for every $a < b$, 
the homology $\M_\ast M(a)$ is supported over the set of simplices $\Delta P_{\leq a} \setminus \Delta P_{< a}$ which is a subset of $\Delta P_{< b}$.

\begin{ex}
\label{ex:first_homology}
Consider the poset $P$ and the two $P$-modules $M, N :  P \to \vec$ in Figure~\ref{fig:first_example}.
The table below lists the M\"obius homology of $M$ and $N$ at each element.
Notice the M\"obius inversion is the Euler characteristic of the M\"obius homology.
    \begin{center}
        \begin{tabular}{L || C | C | C} 
	&  b & a & c \\ \hline \hline
	\M_0 M & \field & \field & \field \\
	\M_1 M & 0 & \field & 0 \\ \hline \hline
	\M_0 N & \field & 0 & \field \\
        \M_1 N & 0 & 0 & 0 
	\end{tabular}
	\end{center}
\end{ex}

\begin{ex}
\label{ex:seconnd_homology}
Consider the poset $P$ and the two $P$-modules 
$M, N :  P \to \End(\C)$ in Figure~\ref{fig:second_example}.
The table below lists the M\"obius homology of $M$ and $N$ at each element.
Notice that the M\"obius inversion is the Euler characteristic of the M\"obius homology.
    \begin{center}
        \begin{tabular}{L || C | C | C} 
	&  b & a & c \\ \hline \hline
	\M_0 M & 
        \begin{pmatrix} \lambda \end{pmatrix} 
        & \begin{pmatrix} \lambda \end{pmatrix} 
        & \begin{pmatrix} \lambda \end{pmatrix} \\
	\M_1 M &  0 & \begin{pmatrix} \lambda \end{pmatrix}  & 0 \\ \hline \hline
	\M_0 N 
        & \begin{pmatrix} \lambda \end{pmatrix}
        & 0 & \begin{pmatrix} \lambda \end{pmatrix} \\
	\M_1 N & 0 & 0 & 0 
	\end{tabular}
	\end{center}
\end{ex}

\begin{ex} \label{ex:cusp}
Consider the cusp map $f: \mathbb{R}^2 \to \mathbb{R}^2$ in Figure~\ref{fig:map}.
The critical values of~$f$ partition the
codomain into five cells: one 0-cell,
two $1$-cells, and two $2$-cells.
Let $P$ be the face poset of this partition
and consider the module $M :  P \to \vec$
in Figure~\ref{fig:cusp_module}
that assigns to every cell its fiberwise zero-dimensional homology to every face relation the natural map between the fibers.
The table below lists the M\"obius homology of~$M$ at each cell.
Notice all non-trivial entries occur
at the codimension of the cell.
M\"obius homology of $M$ identifies the normal component of $f$ to each cell.
\begin{center}
    \begin{tabular}{L || C | C | C | C | C} 
    &  a & b & c &d &e \\ \hline \hline
    \M_0 M & \field^3 & 
    \field & 
    0 &  0 & 0 \\
    \hline
    \M_1 M & 0 & 
    0 &  \field^2 &\field^2 & 0 \\
    \hline
    \M_2 M &  0 &   0 &
    0 & 0 & \field
    \end{tabular}
\end{center}

\begin{figure}
\centering
\begin{subfigure}[b]{0.55\textwidth}
\centering
\includegraphics[width=\textwidth,page=1]{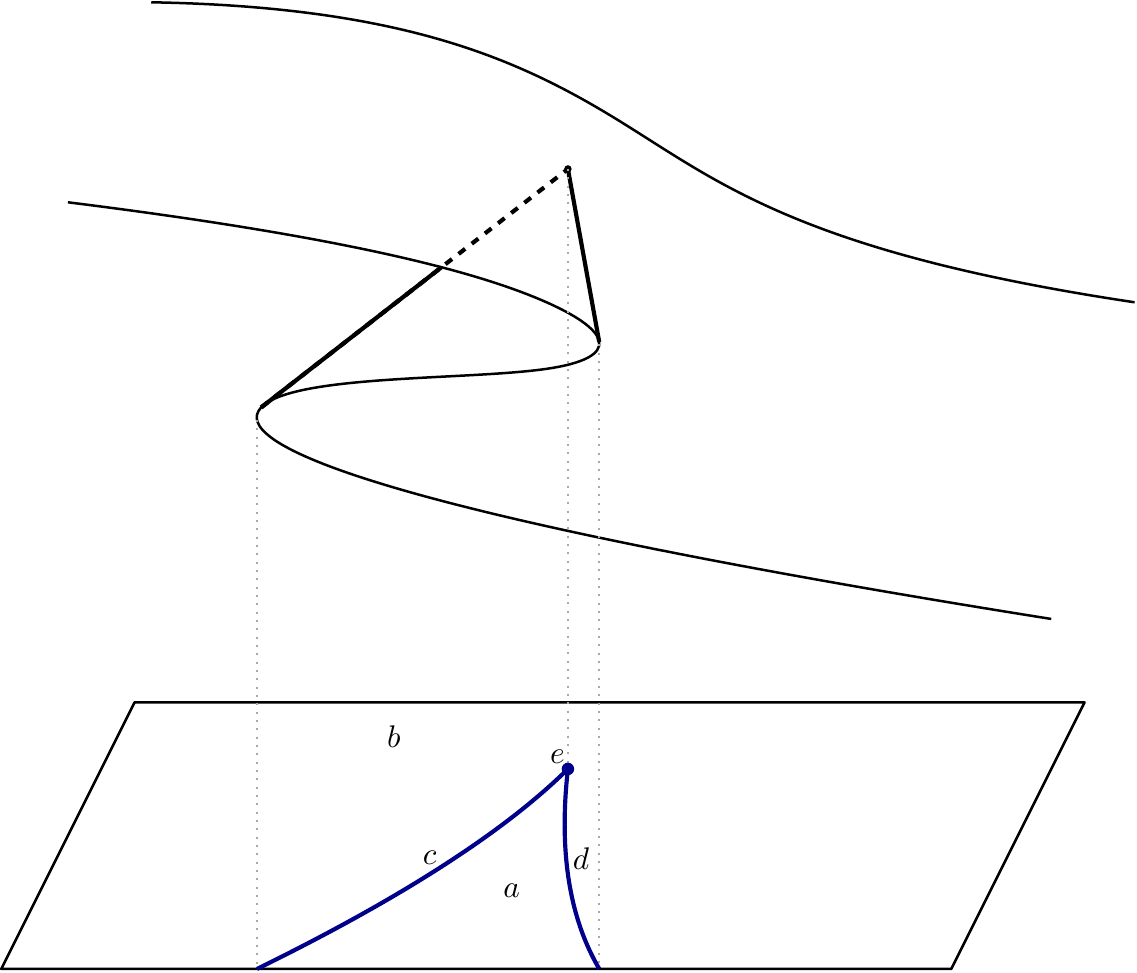}
\caption{Cusp map}
\label{fig:map}
\end{subfigure}
\begin{subfigure}[b]{0.4\textwidth}
\begin{center}
\begin{tikzpicture}[scale=0.7, every node/.style={transform shape}] 
 \draw[ domain=0:3, smooth, variable=\x, blue,thick] plot ({2*\x}, {sqrt(8*\x*\x*\x/27)});
  \draw[ domain=0:3, smooth, variable=\x, blue,thick] plot ({2*\x}, {-sqrt(8*\x*\x*\x/27)});
  \node (c) at (0,0) {};
 \filldraw[blue] (c) circle (2pt);  
 \node (e0) at (6,0) {$k^3$};
 \node (e1) at (-2,0) {$k$};
 \node (c1) at (-0.2,0.3) {$k$};
  \filldraw[white] (3.4,1.2) circle (7pt);  
  \filldraw[white] (3.4,-1.2) circle (7pt);  
  \node (r2) at (3.4,1.2) {$k^2$};
  \node (r3) at (3.4,-1.2) {$k^2$};
  \draw[->] (e0.north west) to[out=135,in=-10]
 node[above=10,pos=0.1] {\scriptsize{$\begin{bmatrix} 1 & 1 & 0 \\ 0 &0 &1 \end{bmatrix} $}    } (r2.east);
   \draw[->] (e0.south west) to[out=-135,in=-10]
 node[below=12,pos=0.1] {\scriptsize{$\begin{bmatrix} 1 & 0 & 0 \\ 0 &1 &1 \end{bmatrix} $}    } (r3.east);
  \draw[->] (e1.north) to[out=90,in=100,looseness=1.5]  
 node[above=2] {\scriptsize{$\begin{bmatrix}  0 \\ 1 \end{bmatrix} $}    } (r2.north);
  \draw[->] (e1.south) to[out=-90,in=-100,looseness=1.5] 
 node[below=2] {\scriptsize{$\begin{bmatrix} 1 \\ 0   \end{bmatrix} $}    } (r3.south);
 \draw[->] (e1.east) to 
 node[above=2] {\scriptsize{$\begin{bmatrix} 1 \end{bmatrix} $}    } (c.west);
 \draw[->] (r2.west) to[out=170,in=70]
 node[above=2] {\scriptsize{$\begin{bmatrix} 1 & 1\end{bmatrix} $}    } (c.north);
 \draw[->] (r3.west) to[out=-170,in=-70]
 node[below=2] {\scriptsize{$\begin{bmatrix} 1 & 1\end{bmatrix} $}    } (c.south);
 %  \draw[->] (e0.west) to
 % node[below=2] {\scriptsize{$\begin{bmatrix} 1 & 1 &1 \end{bmatrix} $}    } (c.east);
\end{tikzpicture}
\end{center}
\caption{$H_0$ module}
\label{fig:cusp_module}
\end{subfigure}
\caption{Module over face poset for the cusp map.}
\label{fig:cusp}
\end{figure}
\end{ex}

\begin{ex} \label{ex:third_mobius}
Consider the modules $M$ and $N$ of subgroups
of $\Z/4\Z$ and $\Z/2\Z \oplus \Z/2\Z$, respectively,
in Figures~\ref{fig:subgroups_one} and~\ref{fig:subgroups_two}.
Both modules are valued in $\FinAb$.
The table below lists the M\"obius homology of $M$
and $N$ at the top element.

    \begin{center}
    \begin{tabular}{L || C | C } 
    &  \frac{\Z}{2\Z} \oplus \frac{\Z}{2\Z} 
    & \frac{\Z}{4\Z} \vspace{2pt}\\  \hline \hline
    \M_0   & 0 & \frac{\Z}{2\Z} \vspace{2pt} \\\hline
    \M_1 & \frac{\Z}{2\Z} & 0 \vspace{2pt} \\ \hline
    \M_2 & 0 & 0
    \end{tabular}
    \end{center}

    \begin{figure}
        \centering
        \begin{subfigure}[b]{0.4\textwidth}
        \begin{equation*}
        \begin{tikzcd}
        & \frac{\Z}{2\Z} \oplus \frac{\Z}{2\Z} & \\
        & & \\
        \frac{\Z}{2\Z} 
        \ar[ruu,"\footnotesize{\begin{pmatrix} 1 \\ 0 \end{pmatrix}}"] & 
             \frac{\Z}{2\Z} 
        \ar[near start,uu, "\footnotesize{\begin{pmatrix} 1 \\ 1 \end{pmatrix}}"]
        & \frac{\Z}{2\Z} 
        \ar[luu,"\footnotesize{\begin{pmatrix} 0 \\ 1 \end{pmatrix}}"'] \\&&\\
        & 0 \ar[luu] \ar[ruu] \ar[uu]&
        \end{tikzcd}
        \end{equation*}
        \caption{$M: P \to \FinAb$}
        \label{fig:subgroups_one}
        \end{subfigure}
        \hfill
        \begin{subfigure}[b]{0.5\textwidth}
        \begin{equation*}
        \begin{tikzcd}
            \frac{\Z}{4\Z} \\ \\
            \frac{\Z}{2\Z} \ar[uu, "\footnotesize{\begin{pmatrix} 2 \end{pmatrix}}"]\\
            \\
            0 \ar[uu]
        \end{tikzcd}
        \end{equation*}
        \caption{$N: Q \to \FinAb$}
        \label{fig:subgroups_two}
        \end{subfigure}
    \caption{Modules of subgroups.}
    \label{fig:third_example}
    \end{figure}
\end{ex}

%%%%%%%%%%%%%%%%%%%%%%%%%%%%%%%%%%%%%%%%%%%%%%%
\subsection{Main Theorem}
%%%%%%%%%%%%%%%%%%%%%%%%%%%%%%%%%%%%%%%%%%%%%%%

The following theorem relates M\"obius inversions to M\"obius homology.

\begin{thm} \label{thm:euler}
Let $m :  P \to K(\cat)$ be the dimension
function of a $P$-module $M : P \to \cat$.
Then for all $b \in P$,
$$\partial m(b) = \revision{\chi \big( \Delta P_{\leq b} , \Delta P_{< b}; \cosheaf{M} \big)  =}\sum_{d \geq 0} (-1)^d \big[ \M_\ast M (b) \big].$$
\end{thm}

The proof is based on the following observation.

\begin{lem}[Philip Hall's Theorem, Prop 3.8.5 \cite{stanley_2011}]  \label{lem:Hall}
For $a \leq b$ in $P$, let $n_d(a,b)$ be the number of chains of length $d$ starting at $a$ and ending at $b$.
Then,
    $$\mu [a,b] = \sum_{\revision{d \geq 0}} (-1)^d \cdot n_d(a,b).$$
\end{lem}

\revision{Note that $n_d(a,b)$ is the number of $d$-simplices in $\Delta P$ with minimal element $a$ and maximal element $b$.
Further, the sum of $n_d(a,b)$, over all $a \leq b$, is the total number of $d$-simplices in~$\Delta P$.}

\begin{proof}[Proof of Theorem~\ref{thm:euler}]
We have
    \begin{align*}
   \partial m(b) 
   &= \sum_{a : a \leq b} m(a) \cdot \mu[a,b] 
       && \text{by Equation~\eqref{eq:inverion_formula}} \\[1ex]
   &= \sum_{a : a \leq b} \big[ M(a) \big] \cdot \mu[a,b] 
       && \text{by Definition~\ref{defn:dimension}} \\[1ex]
   &= \sum_{a : a \leq b} \big[ M(a) \big] \cdot \left( \sum_{d \geq 0} (-1)^d \cdot n_d(a,b) \right)
       && \text{by Lemma~\ref{lem:Hall}} \\[1ex]
   &= \sum_{d \geq 0} \sum_{\substack{\sigma \in \Delta P \\ \max(\sigma) = b \\ \dim(\sigma) = d}} 
       (-1)^d \big[M\left( \min(\sigma) \right) \big] \\[1ex]
   &= \chi \big( \Delta P_{\leq b}, \Delta P_{< b}; \cosheaf{M} \big)
       && \text{by Definition~\ref{defn:reuler}} \\[1ex]
   &= \sum_{d \geq 0} (-1)^d \big[ \M_\ast M(b) \big] 
       && \text{by Proposition~\ref{prop:relative_reuler}}
\end{align*}
\end{proof}

In the M\"obius chain complex $\Ch_\bullet M(b)$, only simplices in $\Delta P_{\leq b} \setminus \Delta P_{< b}$ are relevant. For distinct elements $a$ and $b$, we have:
\begin{equation} \label{eq:intersection}
    (\Delta P_{\leq a} \setminus \Delta P_{< a}) \cap (\Delta P_{\leq b} \setminus \Delta P_{< b}) = \emptyset
\end{equation}
Thus, each simplex $\sigma \in \Delta P$ appears in the M\"obius chain complex for exactly one element of $P$: its maximum element $\max(\sigma)$. This observation leads to the following corollary of our main theorem.

\begin{cor}\label{cor:total_homology}
Let $m : P \to K(\cat)$ be the dimension function of a $P$-module $M : P \to \cat$.
Then,
    $$\sum_{a \in P} \partial m(a) = \revision{\chi \big( \Delta P; \cosheaf{M} \big)  =} \sum_{d \geq 
 0} (-1)^d \big[ H_d ( \Delta P; \cosheaf{M} ) \big ].$$
\end{cor}
\begin{proof}
\revision{
We have
    \begin{align*}
    \sum_{a \in P} \partial m(a) 
    &= \sum_{a \in P} \chi(\Delta P_{\leq a}, \Delta P_{< a}; \cosheaf{M}) 
        && \text{by Theorem~\ref{thm:euler}} \\[1ex]
    &= \sum_{a \in P} \sum_{d \geq 0} \sum_{\substack{\sigma \in \Delta P \\ \max(\sigma) = a \\ \dim(\sigma) = d}} 
        (-1)^d \big[M\left( \min(\sigma) \right) \big] 
        && \text{by Definition~\ref{defn:reuler}} \\[1ex]
    &= \sum_{d \geq 0} \sum_{\substack{\sigma \in \Delta P \\ \dim(\sigma) = d}} 
        (-1)^d \big[M\left( \min(\sigma) \right) \big] 
        && \text{by Equation~\eqref{eq:intersection}} \\[1ex]
    &= \chi(\Delta P; \cosheaf{M}) 
        && \text{by Definition~\ref{defn:total_euler}} \\[1ex]
    &= \sum_{d \geq 0} (-1)^d \big[ H_d ( \Delta P; \cosheaf{M} ) \big]
        && \text{by Proposition~\ref{prop:cosheaf_euler}}
\end{align*}
}
\end{proof}

We have seen how M\"obius homology is a richer invariant than the M\"obius inversion.
However, M\"obius homology is not a complete invariant of a $P$-module.

\begin{ex}
Choose an element $\mu \in \field$ and consider the module $M_\mu$ in Figure~\ref{fig:buchet_escolar}.
For the choice of a second element $\nu \in \field$, the two modules
$M_\mu$ and $M_\nu$ are not isomorphic~\cite{BuchetE22}.
However, by direct calculation one can show that the element-wise M\"obius homology groups 
for the two modules are isomorphic.
Thus, M\"obius homology is not a complete invariant.
\begin{figure}
\begin{equation*}
\begin{tikzcd}[ampersand replacement=\&]
    \field  \ar[rr,"\footnotesize\begin{pmatrix} 1 \\ 0 \end{pmatrix}"] 
    \&\& \field^2 \ar[rr, "\footnotesize\begin{pmatrix} 1 \end{pmatrix}"] 
    \&\& \field^2 
    \ar[rr, "\footnotesize{\begin{pmatrix} 1 & 0 \end{pmatrix}}"]
    \&\& \field \ar[rr]
    \&\&  0 \\  \\
    0 \ar[rr] \ar[uu] \&\& \field \ar[rr, "\footnotesize{\begin{pmatrix} 0 \\ 1 \end{pmatrix}}"]
    \ar[uu, "\footnotesize{\begin{pmatrix} 1 \\ \mu \end{pmatrix}}"]
    \&\&  \field^2 \ar[rr, "\footnotesize{\begin{pmatrix} 1 \end{pmatrix}}"]
    \ar[uu, "\footnotesize{\begin{pmatrix} 1  & 1 \\ 1 & \mu \end{pmatrix}}"]
    \&\& \field^2 \ar[rr, "\footnotesize{\begin{pmatrix} 0 & 1 \end{pmatrix}}"]
    \ar[uu, "\footnotesize{\begin{pmatrix} 1 & 1 \end{pmatrix}}"]
    \&\& \field \ar[uu]
\end{tikzcd}
\end{equation*}
\caption{$M_\mu$ is a $P$-module in $\vec$ for every $\mu \in \field$.}
\label{fig:buchet_escolar}
\end{figure}
\end{ex}

%%%%%%%%%%%%%%%%%%%%%%%%%%%%%%%%%%%%%%%%%%%%%%%
%%%%%%%%%%%%%%%%%%%%%%%%%%%%%%%%%%%%%%%%%%%%%%%
\section{Rota's Galois Connection Theorem}
%%%%%%%%%%%%%%%%%%%%%%%%%%%%%%%%%%%%%%%%%%%%%%%
%%%%%%%%%%%%%%%%%%%%%%%%%%%%%%%%%%%%%%%%%%%%%%%
This section establishes a homological version of Rota's Galois connection theorem (Theorem~\ref{thm:rota_homology})
and explores two immediate consequences.
We start with Rota's original statement
relating the two M\"obius functions involved in a Galois connection.

\begin{thm}[Rota's Galois Connection Theorem~\cite{gian1964foundations}]
\label{thm:rota_original}
Let $P$ and $Q$ be finite posets and $f : P \leftrightarrows Q : g$ a Galois connection.
Then for any $a \in P$ and $x \in Q$,
	$$\sum_{b: f(b) = x} \mu_P(a, b) = \sum_{y: g(y) = a} \mu_Q (x, y).$$
\end{thm}

The following statement by G\"ulem and McCleary
relates the two M\"obius inversions involved in a Galois connection and is equivalent to 
Theorem~\ref{thm:rota_original}
when the functions involved are valued in \revision{any abelian group}.

\begin{thm}[Rota's Galois Connection Theorem~\cite{GulenMcCleary}]
\label{thm:rota}
Let $m : P \to \Z$ and $n : Q \to \Z$ be two functions and $f : P \leftrightarrows Q :  g$
a Galois connection. 
If $n = g \circ m$, then for all $y \in Q$,
$$\partial n(y)  = \sum_{b \in f^{-1}(y)} \partial m(b).$$
\end{thm}

Given a Galois connection $f : P \leftrightarrows Q : g$ and an element $y \in Q$, consider the following two subcomplexes of $\Delta P$:
    \begin{align*}
    \Delta f^{-1}_{\leq y} := \big\{\sigma \in \Delta P: f(\max(\sigma)) \leq y \big \}
    &&
    \Delta f^{-1}_{< y} := \big\{\sigma \in \Delta P: f(\max(\sigma)) < y \big \}.
    \end{align*}
We now state our theorem and argue how it categorifies Rota's Galois connection theorem.
The proof of the theorem is presented in the following subsection.

%Note that every integral function  $m : P \to \Z$ is the dimension function of a canonical $P$-module $M : P \to \vec$ where $M(a) := \field^{m(a)}$ and all the morphisms are zero. 

\begin{thm}[Homological Rota's Galois Connection Theorem]
\label{thm:rota_homology}
For a module $M : P \to \cat$ and a Galois connection $f : P \leftrightarrows Q : g$, let 
the module $N := M \circ g: Q \to \cat$.
For every $y \in Q$, there is a canonical isomorphism
    \begin{equation}
    \label{eq:galois_homology}
    \M_\ast N (y) 
    \cong H_\ast \big( \Delta f^{-1}_{\leq y}, \Delta f^{-1}_{< y}; \cosheaf{M} \big).
    \end{equation}
\end{thm}

Theorem~\ref{thm:rota_homology} categorifies Theorem~\ref{thm:rota} as follows.
Let $m, n : \to K(\cat)$ be the dimension functions of the modules $M$ and $N$, respectively.
By Theorem~\ref{thm:euler}, 
 $$\partial n(y) = 
 \sum_{d\geq 0} \big[ \M_d N (y) \big].$$
Note that the preimage on the
right side of Equation~\eqref{eq:galois_homology}
is the following disjoint union:
$$\Delta f^{-1}_{\leq y} \setminus \Delta f^{-1}_{< y} = \bigsqcup_{b \in f^{-1}(y)} \Delta P_{\leq b} \setminus \Delta P_{< b}.$$
This combined with Definition~\ref{defn:reuler} and Lemma~\ref{prop:relative_reuler} implies 
$$\sum_{b \in f^{-1}(y)} \partial m(b) =
\sum_{d \geq 0}(-1)^d \Big[ H_\ast \big( \Delta f^{-1}_{\leq y}, \Delta f^{-1}_{< y}; \cosheaf{M} \big) \Big].$$

\begin{rmk}
The right side of Equation~\eqref{eq:galois_homology} can be seen as the limiting term of a Leray-type spectral sequence.
This spectral sequence accounts for the differentials between the M\"obius homology objects $\big\{ \M_d M (b)  \big\}$ over all $b \in f^{-1}(y)$ and dimensions~$d$.
\end{rmk}

\paragraph{Refinements}
An immediate consequence of Theorem~\ref{thm:rota_homology}
is that M\"obius homology is invariant to refinements of the underlying poset as follows.
A \emph{refinement} of a poset~$P$ is a Galois connection $f : P \leftrightarrows Q : g$
such that $f$ is injective or, equivalently,
$g$ is surjective.
In this case, we think of $P$ as a subposet of $Q$.

\begin{cor}
Given a refinement $f : P \leftrightarrows Q  : g$ and a $P$-module $M$, let
the $Q$-module $N := M \circ g$.
For every $b \in P$,
$\M_\ast M (b) \cong 
\M_\ast N (b)$.
For every $y \in Q \setminus P$, $\M_\ast N (y) = 0$ because $\Delta f^{-1}_{\leq y} \setminus \Delta f^{-1}_{< y} = \emptyset$.
\end{cor}

%%%%%%%%%%%%%%%%%%%%%%%%%%%%%%%%%%%%%%%%%%%%%%
\paragraph{Bounding Homological Dimension}
%%%%%%%%%%%%%%%%%%%%%%%%%%%%%%%%%%%%%%%%%%%%%%
When the underlying poset of a $P$-module $M$ is a finite distributive lattice, 
we show in Corollary~\ref{cor:depth} that there exists a bound on the highest non-trivial dimension of its M\"obius homology.

Recall that a \textit{lattice} is a poset in which every pair of elements has a meet (greatest lower bound), denoted $\wedge$, 
and a join (least upper bound), denoted $\vee$. 
A lattice is \textit{distributive} if, for all $a, b, c \in P$, 
\[
   a \vee (b \wedge c) = (a \vee b) \wedge (a \vee c).
\]
If $a, c \in P$ with $a < c$, we say that \emph{$c$ covers $a$}, denoted $a \lessdot c$, 
if there is no $b$ such that $a < b < c$. 
An element $a \in P$ is \emph{meet-irreducible} if it is not the maximum element and 
there are no distinct elements $b, c \neq a$ such that $a = b \wedge c$.

Now, fix an element $b \in P$, and let $C_b$ be the set of all elements of $P$ covered by $b$:
\[
  C_b := \{b' \in P : b' \lessdot b\}.
\]
The set $C_b$ generates a sublattice of $P$, called the 
\emph{meet-generated sublattice}:
\[
   \mathcal{M}(C_b) := \{b\} \cup \{ a \in P : a = b_1 \wedge \cdots \wedge b_k 
   \; \text{for some}\; b_1, \dots, b_k \in C_b \}.
\]
Observe that $\mathcal{M}(C_b)$ has a minimum element, denoted $0_b$, which is the meet of all elements in~$C_b$, and its maximum element is $b$. 
Additionally, the set of meet-irreducible elements of $\mathcal{M}(C_b)$ is precisely $C_b$.

For every $b \in P$, there is a canonical Galois connection 
\[
   f : P_{\leq b} \leftrightarrows \mathcal{M}(C_b) : g
\]
defined as follows: for $a \in P_{\leq b}$, let $A = \{b\} \cup \{ b' \in C_b : a \leq b' \}$. 
Define $f(a) = \wedge A$, and let $g$ be the inclusion map. 
It is straightforward to verify that we indeed have a Galois connection. 
See Figure~\ref{fig:dim_collapse} for an illustration of this Galois connection.

Before stating and proving the corollary, we need one additional observation. 
In a distributive lattice, every maximal chain has the same length, 
and that length is given by the number of meet-irreducible elements \cite[Lemma~2, p.~59]{birkhoff1940lattice}. 
Consequently, the length of every maximal chain between $0_b$ and $b$ in $\mathcal{M}(C_b)$ equals the cardinality of $C_b$.

\begin{cor}\label{cor:depth}
Let $P$ be a finite distributive lattice, $b \in P$, and $n$ the number of elements covered by $b$.
Then, for any module $M : P \to \cat$, $\M_{> n} M(b) = 0.$
\end{cor}

\begin{proof}
Let $f : P_{\leq b} \leftrightarrows \mathcal{M}(C_b) : g$ be the Galois connection defined above. 
Note that $f^{-1}(b) = \{ b \}$. Let $N : \mathcal{M}(C_b) \to \cat$ be the module defined by $N = M \circ g$. 
Since the length of any chain in $\mathcal{M}(C_b)$ is at most~$n$, we have $\M_{> n} N(b) = 0$. 
The desired result follows from Theorem~\ref{thm:rota_homology}.
\end{proof}

\begin{figure}
\centering
\includegraphics[width=0.5\textwidth,page=3]{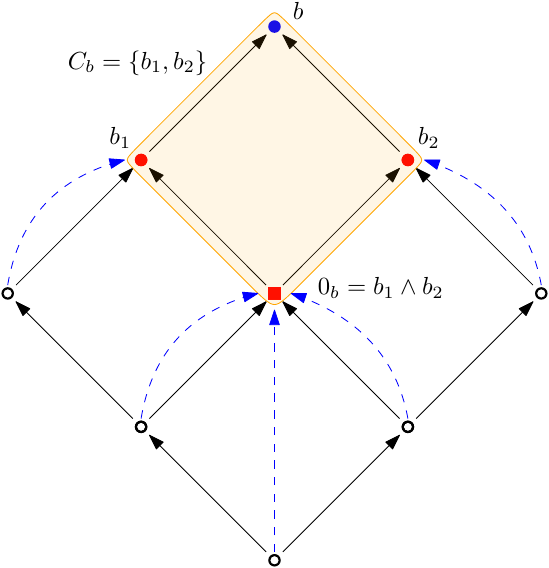}
\caption{An example of a Galois connection $f : P_{\leq b} \leftrightarrows \mathcal{M}(C_b) : g$ on a distributive lattice $P$, where $f$ is represented by blue dashed arrows and $g$ is the inclusion map.}
\label{fig:dim_collapse}
\end{figure}

%%%%%%%%%%%%%%%%%%%%%%%%%%%%%%%%%%%%%%%%%%%%%%%
\subsection{Proof of Theorem~\ref{thm:rota_homology}}
\label{sec:proof}
%%%%%%%%%%%%%%%%%%%%%%%%%%%%%%%%%%%%%%%%%%%%%%%
The proof establishes a chain homotopy equivalence between 
\[ 
C_\bullet \big( \Delta Q_{\leq y}, \Delta Q_{< y}; \cosheaf{N} \big) 
\quad \text{and} \quad 
C_\bullet \big( \Delta f^{-1}_{\leq y}, \Delta f^{-1}_{< y}; \cosheaf{M} \big).
\] 
The main idea is to extend a relative version of Proposition~\ref{prop:chain_eq}, incorporating the techniques from Lemmas~\ref{lem:left_chain_homotopy} and \ref{lem:right_chain_homotopy}, to the cosheaf setting.

The simplicial map $\Delta f$ induces a chain map 
$$\Delta \hat f_\bullet  : C_\bullet \big( \Delta f^{-1}_{\leq y}, \Delta f^{-1}_{< y}; \cosheaf{M} \big)
\to C_\bullet \big( \Delta Q_{\leq y},  \Delta Q_{< y}; \cosheaf{N} \big)$$ 
generated simplex-wise as follows.
For an $n$-simplex $\sigma \in \Delta f^{-1}_{\leq y}$, let $\tau := \Delta f (\sigma)$.
By Lemma~\ref{lem:galois}~(i),
$\min(\sigma) \leq g (\min \tau)$ in~$P$.
If $\tau$ is an $n$-simplex, define the restriction
$\Delta \hat f_n |_\sigma:
\cosheaf{M}(\sigma) \to \cosheaf{N}(\tau)$
as the morphism $M \big(\min(\sigma) \leq g( \min \tau) \big).$
If $\tau$ is not an $n$-simplex, then $\Delta \hat f_n |_\sigma$ is the zero morphism.
One easily checks that $\Delta \hat f_\bullet$ is indeed a chain map.

The simplicial map $\Delta g$ induces a chain map 
$$\Delta \hat g_\bullet  : C_\bullet \big( \Delta Q_{\leq y},  \Delta Q_{< y}; \cosheaf{N} \big) \to C_\bullet \big( \Delta f^{-1}_{\leq y}, \Delta f^{-1}_{< y}; \cosheaf{M} \big)$$
induced simplex-wise as follows.
For an $n$-simplex $\sigma \in \Delta Q_{\leq y}$, let $\tau := \Delta g (\sigma)$.
Note 
$g(\min(\sigma)) = \min \tau$.
If $\tau$ is an $n$-simplex, define the restriction
$\Delta \hat g_n |_\sigma : \cosheaf{N}(\sigma) \to \cosheaf{M}(\tau)$
as the identity morphism $M\big( g(\min(\sigma)) = \min \tau\big)$.
If $\tau$ is not an $n$-simplex, then 
$\Delta \hat g_n |_\sigma$ is the zero morphism.
One easily checks that $\Delta \hat g_\bullet$ is indeed a chain map.

\begin{lem}\label{lem:left_cosheaf_homotopy}
The composition $\Delta \hat g \circ \Delta \hat f : C_\bullet \big( \Delta f^{-1}_{\leq y}, \Delta f^{-1}_{< y}; \cosheaf{M} \big) \to C_\bullet \big(\Delta f^{-1}_{\leq y}, \Delta f^{-1}_{< y}; \cosheaf{M} \big)$ is chain homotopic to the identity.
\end{lem}
\begin{proof}
The chain homotopy $\phi_n$ described by Equation~\eqref{eq:left_chain_homotopy} induces
a chain homotopy 
$$\Phi_n : C_n \big( \Delta f^{-1}_{\leq y}, \Delta f^{-1}_{< y}; \cosheaf{M} \big) 
\to C_{n+1}\big( \Delta f^{-1}_{\leq y}, \Delta f^{-1}_{< y}; \cosheaf{M} \big)$$
defined simplex-wise as follows.
For notational convenience, substitute $\sigma$ and $\tau_i$ for the simplices on the left and right sides of Equation~\eqref{eq:left_chain_homotopy}:
	\begin{equation*}
	\phi_n (\sigma) := 
        \sum_{i=0}^n (-1)^i \tau_i .
	\end{equation*}
Note that $\min(\sigma) = \min \tau_i$,
for all $i$.
Thus $\Phi_n$ restricted to the summand $\cosheaf{M}(\sigma)$ is the direct sum of signed identity morphisms: 
$$\Phi_n |_\sigma := \bigoplus_{i=0}^n (-1)^i M(\min(\sigma) = \min \tau_i).$$
By Lemma~\ref{lem:left_chain_homotopy},
$\Phi_n$ is indeed a chain homotopy.
\end{proof}

\begin{lem}\label{lem:right_cosheaf_homotopy}
The composition $\Delta \hat f \circ \Delta \hat g :  C_\bullet \big( \Delta Q_{\leq y},  \Delta Q_{< y}; \cosheaf{N} \big) \to C_\bullet \big( \Delta Q_{\leq y},  \Delta Q_{< y}; \cosheaf{N} \big)$ is chain homotopic to the identity.
\end{lem}
\begin{proof}
The chain homotopy $\psi_n$ described by Equation~\eqref{eq:right_chain_homotopy} induces
a chain homotopy 
$$\Psi_n : C_n \big( \Delta Q_{\leq y},  \Delta Q_{< y}; \cosheaf{N} \big) \to C_{n+1}\big( \Delta Q_{\leq y},  \Delta Q_{< y}; \cosheaf{N} \big)$$
defined simplex-wise as follows.
For notational convenience, substitute $\sigma$ and $\tau_i$ for the simplices on the left and right sides of Equation~\eqref{eq:right_chain_homotopy}:
	\begin{equation*}
	\psi_n (\sigma) := \sum_{i=0}^n (-1)^i \tau_i .
	\end{equation*}
By Lemma~\ref{lem:galois}~(i), $g(\min(\sigma)) = g(\min \tau_i)$ implying $\cosheaf{N}(\sigma) = \cosheaf{N}(\tau_i)$.
Thus $\Psi_n$ restricted to the summand $\cosheaf{N}(\sigma)$ is the direct sum of signed identity morphisms
    $$\Psi_n |_\sigma := \bigoplus_{i=0}^n (-1)^i M \big(g (\min(\sigma)) = g(\min \tau_i)\big).$$
By Lemma~\ref{lem:right_chain_homotopy},
$\Psi_n$ is indeed a chain homotopy.
\end{proof}

We now complete the proof.
If $\sigma \in \Delta f^{-1}_{< y}$, then of course $f (\max(\sigma)) < y$.
By Lemma~\ref{lem:galois}~(iii),
$\Delta g \circ \Delta f (\sigma) \in \Delta f^{-1}_{< y}$.
Thus $\Delta \hat g \circ \Delta \hat f$ passes
to a chain endomorphism:
$$\frac{C_\bullet \big(\Delta f^{-1}_{\leq y}; \cosheaf{M} \big)}{C_\bullet \big( \Delta f^{-1}_{< y}; \cosheaf{M} \big)} \to
\frac{C_\bullet \big(\Delta f^{-1}_{\leq y}; \cosheaf{M} \big)}{C_\bullet \big( \Delta f^{-1}_{< y}; \cosheaf{M} \big)}.$$
Now consider the chain homotopy $\phi_n$ in Equation~\eqref{eq:left_chain_homotopy}.
By the same reasoning, if $a_0  < \cdots < a_n$
is in $\Delta f^{-1}_{< y}$,
then every simplex on the right side of the equation is also in $\Delta f^{-1}_{< y}$.
Thus the chain homotopy $\Phi_n$ passes
to a relative chain homotopy.
Now for the other side.
If $\tau \in \Delta Q_{< y}$, then $\max \tau < y$.
By Lemma~\ref{lem:galois}~(ii), 
$\Delta f \circ \Delta g (\tau) \in 
\Delta Q_{<  y}$.
Thus $\Delta \hat f \circ \Delta \hat g$ passes to a chain endomorphism on
$$\frac{C_\bullet (\Delta Q_{\leq y}; \cosheaf{N})}{C_\bullet (\Delta Q_{< y}; \cosheaf{N})} \to \frac{C_\bullet (\Delta Q_{\leq y}; \cosheaf{N})}{C_\bullet (\Delta Q_{< y}; \cosheaf{N})}.$$
Now consider the chain homotopy $\psi_n$ in Equation~\eqref{eq:right_chain_homotopy}.
By the same reasoning, if $x_0  < \cdots < x_n$
is in $\Delta Q_{< y}$,
then every simplex on the right side of the equation is also in $\Delta Q_{< y}$.
Thus the chain homotopy $\Psi_n$ passes
to a relative chain homotopy.

%%%%%%%%%%%%%%%%%%%%%%%%%%%%%%%%%%%%%%%%%%%%%%%
%%%%%%%%%%%%%%%%%%%%%%%%%%%%%%%%%%%%%%%%%%%%%%%
\section{Persistent Homology}
\label{sec:persistence}
%%%%%%%%%%%%%%%%%%%%%%%%%%%%%%%%%%%%%%%%%%%%%%%
%%%%%%%%%%%%%%%%%%%%%%%%%%%%%%%%%%%%%%%%%%%%%%%
We are now ready to study the concept of \emph{persistence} for a \(P\)-module using M\"obius homology. 
The goal of persistence is to capture information that persists across subsets of \(P\). 
The simplest subsets to consider are (half-open) intervals of \(P\).
Our primary construction is the \emph{birth-death module} associated with any given \(P\)-module. 
When \(P\) is totally ordered, the M\"obius homology of this module categorifies the traditional persistence diagram.
To match classical persistence, we assume that the poset \(P\) has a maximum element, denoted \(\infty\). 
This ensures we can discuss features that persist indefinitely—a central theme in persistent homology.

For $a \in P$ and an object $X \in \ob \cat$, let 
$X^{\uparrow a} : P \to \cat$ be the following module.
For every $b \in P$, 
    \begin{equation*}
    X^{\uparrow a}(b) :=
    \begin{cases}
    X & \text{ if $a \leq b$} \\
    0 & \text{ otherwise}
    \end{cases}
    \end{equation*}
and for every $b \leq c$ the morphism $X^{\uparrow a}(b \leq c)$ is the identity when $a \leq b$ and zero otherwise.
A $P$-module is \emph{free} if it is the direct sum of modules of the form $X^{\uparrow a}$.

\begin{defn}
    A \define{free presentation} of a $P$-module $M  : P \to \cat$ is a natural transformation $\phi : F \Rightarrow M$ from a free $P$-module~$F$ such that for all $a \in P$ the morphism $\phi(a) : F(a) \to M(a)$ is an epimorphism.
\end{defn}

The construction of a free presentation below appeared for $\vec$ in \cite{GulenMcCleary}. Here we present the same construction in an abelian category for completeness. 

\begin{prop}
Every $P$-module $M$ has a free presentation.
\end{prop}
\begin{proof}
Let $F$ be the free $P$-module $\bigoplus_{a \in P} M(a)^{\uparrow a}$.
Define the natural transformation $\phi : F \Rightarrow M$
summand-wise as follows.
For $a \in P$, let $\phi_a : M(a)^{\uparrow a} \Rightarrow M$ be the following natural transformation:
    \begin{equation*}
    \phi_a(b) :=
    \begin{cases}
    M(a \leq b) & \text{if $a \leq b$} \\
    0 & \text{otherwise.}
    \end{cases}
    \end{equation*}
Note that $\phi_a(a)$ is an isomorphism.
Then $\phi$ is the direct
sum $\oplus_{a\in P} \phi_a$.
\end{proof}

Every module $M$ has a free presentation, but it is not unique.
For now, we develop persistence for a fixed free presentation and show it is
well defined in the next subsection.

\begin{defn}
The \define{interval poset} of a poset \(P\), denoted \(\Int P\), is the set of all formal half-open intervals \([a, b)\) with \(a \leq b\). The ordering on \(\Int P\) is defined by \([a, b) \preceq [c, d)\) if and only if \(a \leq c\) and \(b \leq d\). See Figure~\ref{fig:poset_int}.
The \define{diagonal} is the subposet of $\Int P$ consisting of intervals of the form $[a,a)$.
\end{defn}

\begin{rmk}
Our definition of $\Int P$ is designed so that our theory of persistence aligns with classical persistence when $P$ is totally ordered. In classical persistence, one typically studies poset modules of the form $\mathbb{R} \to \vec$, where $\mathbb{R}$ is the real numbers with the usual total ordering. Although $\mathbb{R}$ is an infinite poset, the module itself changes isomorphism type at only finitely many real numbers. Therefore, such a module can be modeled as a $P$-module, where $P$ is a finite, totally ordered poset.

The classical persistence diagram is a multiset of open intervals of the form $[a, b)$, along with points on the diagonal. Our definition of $\Int P$ ensures that when a classical persistence module is modeled in our purely combinatorial framework, the resulting persistence diagram matches the classical persistence diagram.
See Examples~\ref{ex:classical_pd_bd} and~\ref{ex:classical_pd_ker}.
\end{rmk}

    \begin{figure}
    \centering
    \begin{subfigure}[c]{0.3\textwidth}
    \centering
    \begin{equation*}
    \begin{tikzcd}
    b \ar[r] & \infty \\
    a \ar[u] \ar[r] & c \ar[u]
    \end{tikzcd}
    \end{equation*}
    \caption{$P$}
    \end{subfigure} \hfill
    \begin{subfigure}[c]{0.6\textwidth}
    \centering
    \begin{equation*}
    \begin{tikzcd}
    & {[\infty,\infty)} & \\
    {[b,\infty)} \ar[ur] & & {[c,\infty)} \ar[ul] \\
    {[b,b)} \ar[u]  & {[a,\infty)} \ar[ru] \ar[lu] & {[c,c)} \ar[u] \\
    {[a,b)} \ar[u] \ar[ur] && {[a,c)} \ar[u] \ar[ul] \\
    & {[a,a)} \ar[lu] \ar[ru] &
    \end{tikzcd}
    \end{equation*}
    \caption{$\Int P$}
    \end{subfigure}
    \caption{Poset (a) and its interval poset (b).}
    \label{fig:poset_int}
    \end{figure}

Fix a free presentation $\phi : F \Rightarrow M$. We are interested in two objects associated with every interval $[a, b)$, depending on whether $b = \infty$.
First, assume $b \neq \infty$ and consider the following pullback diagram:
\[
\begin{tikzcd}
F(a) \cap \ker \phi_b \ar[rr, dashrightarrow] \ar[d, dashrightarrow]
&& \ker \phi(b) \ar[d, hookrightarrow] \\
F(a) \ar[rr, hookrightarrow] 
&& F(b).
\end{tikzcd}
\]
We interpret this pullback object as representing generators that arise by the start of the interval, namely~$a$, and become relations by the end of the interval, namely~$b$.
Second, assume $b = \infty$ and consider the object $F(a)$. We interpret $F(a)$ as representing generators that arise by the start of the interval, namely~$a$, and are forced to become relations by $\infty$.

For every $[a,b) \preceq [c,d)$, the following commutative diagram shows the existence of canonical monomorphisms between the two objects of interest described above 
depending on whether $b$ or $d$ is $\infty$:
\begin{equation*}
\begin{tikzcd}
F(a) \ar[d, hookrightarrow]
\ar[rr, hookrightarrow] && F(b) 
\ar[d, hookrightarrow]
&& \ker \phi(b) 
\ar[ll, hookrightarrow]
\ar[d, hookrightarrow] \\
F(c) \ar[rr, hookrightarrow] && F(d) && \ker \phi(d). \ar[ll, hookrightarrow]
\end{tikzcd}
\end{equation*}

\begin{defn}
The \define{birth-death module} of a free presentation $\phi : F \Rightarrow M$ is the module 
$\BD \phi : \Int P \to \cat$
that assigns to every interval $[a,b)$ (for $b \neq \infty$) the object 
$F(a) \cap \ker \phi_b$, and to every interval $[a,\infty)$, the object $F(a)$. 
For every $[a,b) \preceq [c,d)$, the morphism 
$\BD \phi \big( [a,b) \preceq [c,d) \big)$ is one of the two canonical monomorphisms mentioned above.
\end{defn}

The birth-death module satisfies an interesting pullback property that will be useful in the proof
of Proposition~\ref{prop:classical_persistence}.

\begin{prop} \label{prop:pullback}
Let $\BD \phi$ be the birth-death module of a free presentation $\phi : F \Rightarrow M$.
For every \(a \leq b \leq c \leq d\) in \(P\), the following subdiagrams of \(\BD \phi\) are pullback diagrams:

\begin{equation*}
\begin{tikzcd}
\BD \phi  [a,d)   \ar[r, hookrightarrow] & \BD \phi [b,d)  \\
\BD \phi  [a,c)  \ar[u, hookrightarrow, dashed] \ar[r, hookrightarrow, dashed] & 
\BD \phi [b,c)  \ar[u, hookrightarrow].
\end{tikzcd}
\end{equation*}
\end{prop}

\begin{proof}
The statement follows from a straightforward chase on the following diagram:
\begin{equation*}
\begin{tikzcd}
F \ar[d, Rightarrow, "\phi"] & F(a) \ar[d, twoheadrightarrow] \ar[r, hookrightarrow] & F(b) \ar[d, twoheadrightarrow] \ar[r, hookrightarrow] & F(c) \ar[d, twoheadrightarrow] \ar[r, hookrightarrow] & F(d) \ar[d, twoheadrightarrow] \\
M & M(a) \ar[r] & M(b) \ar[r] & M(c) \ar[r] & M(d).
\end{tikzcd}
\end{equation*}

\end{proof}

Let $\bd \phi : \Int P \to K(\cat)$ denote the dimension function of the birth-death module $\BD \phi$. When $P$ is totally ordered, the M\"obius inversion $\partial (\bd \phi)$ coincides with the classical persistence diagram for all intervals $[a, b)$ with $a < b$. See~\cite{GulenMcCleary} for a proof and Example~\ref{ex:classical_pd_bd}. This connection motivates the following generalization of the classical persistence diagram, as introduced by G\"ulen and McCleary~\cite{GulenMcCleary}.

\begin{defn}
Let $\phi : F \Rightarrow M$ a free presentation of a $P$-module $M$.
The \define{persistence diagram} of $\phi$ is the M\"obius inversion $\partial (\bd \phi)$.
\end{defn}

There is a drawback to this generalization of the persistence diagram.
Bottleneck stability, as introduced by Cohen-Steiner, Edelsbrunner, and Harer~\cite{stability-persistence}, is an important property of the classical persistence diagram, but it
is unlikely such a theorem exists for this generalization.
The big hurdle here is the fact that $\partial (\BD \phi)$ may take negative values when $P$ is not totally ordered.
We hope the order cosheaf of $\BD \phi$ will offer a path forward since, 
by Theorem~\ref{thm:euler}, $\partial (\bd \phi)$ is the Euler charactersictic of $\M_\ast (\BD \phi)$.
In other words, there is now a lift of $\partial (\bd \phi)$ to a topological object and we hope this will inspire a topological bottleneck stability result.

\begin{defn}
Let $\phi : F \Rightarrow M$ be a free presentation of a $P$-module $M$.
The \define{persistent M\"obius homology} of
$\phi$ is the M\"obius homology module $\M_\ast (\BD \phi) : \Int P \to \cat$.
\end{defn}

The following example compares two similar modules over a totally ordered poset.
While their dimensions functions are identical, their persistence diagrams are different.

\begin{ex} \label{ex:classical_pd_bd}
We consider two examples over the poset $P = \{1 < 2 < 3 < 4 < \infty\}$. We construct two $P$-modules $M, N : P \to \vec$, compute their persistence diagrams, and examine their persistent M\"obius homology.
For any $a \leq b$, there is a natural inclusion $\field^{\uparrow b} \hookrightarrow \field^{\uparrow a}$. Denote by $I_{[a, b)}$ the $P$-module $\frac{\field^{\uparrow a}}{\field^{\uparrow b}}$. 

Let $M := I_{[0,4)} \oplus I_{[1,3)} \oplus I_{[2,\infty)}$. A free presentation of $M$ is given by $\phi : F \Rightarrow M$, where:
\[
F := \field^{\uparrow 0} \oplus \field^{\uparrow 1} \oplus \field^{\uparrow 2},
\]
and $\phi$ is the quotient natural transformation.

Let $N := I_{[0,3)} \oplus I_{[1,4)} \oplus I_{[2,\infty)}$. A free presentation of $N$ is given by $\psi : G \Rightarrow N$, where:
\[
G := \field^{\uparrow 0} \oplus \field^{\uparrow 1} \oplus \field^{\uparrow 2},
\]
and $\psi$ is the quotient natural transformation.

Although the dimension functions $m, n : P \to \Z$ of the two modules are identical, their persistence differs significantly. Figure~\ref{fig:classical_bd} illustrates the dimension functions $\bd \phi$ and $\bd \psi$, along with the zero-th M\"obius homology $\M_0 \BD \phi$ and $\M_0 \BD \psi$. In both cases, the M\"obius homology vanishes in all dimensions greater than zero. 
The persistence diagrams $\partial (\bd \phi)$ and $\partial (\bd \psi)$ are interpreted as the Euler characteristics of $\M_\ast \BD \phi$ and $\M_\ast \BD \psi$, respectively.
\end{ex}

\begin{figure}
\begin{center}
\begin{tikzpicture}
        \node[anchor=south west, inner sep=0] (image) at (0,0) {\includegraphics[width=0.9\textwidth,page=9]{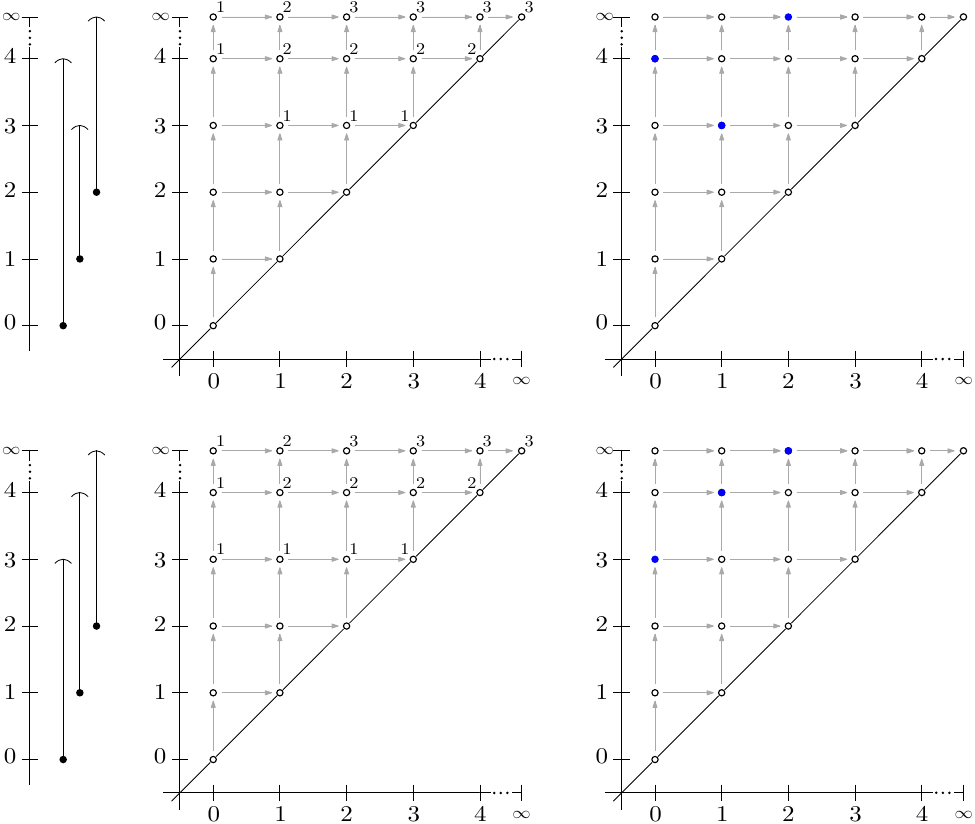}}; %
        \begin{scope}[x={(image.south east)}, y={(image.north west)}] 
            \node[anchor=center] (M1)  at (-0.05,0.78) {$M$};
            \node[anchor=center] (M2)  at (-0.05,0.255) {$N$};
            \node[anchor=center] (a)  at (0.07,-0.04) {(a)};
            \node[anchor=center] (b)  at (0.355,-0.04) {(b)};
            \node[anchor=center] (c)  at (0.81,-0.04) {(c)};
             \node[anchor=center] (n1)  at (0.45,0.65) {$\bd \phi$};
            \node[anchor=center] (n2)  at (0.45,0.15) {$\bd \psi$};
              \node[anchor=center] (n1)  at (0.9,0.65) {\textcolor{blue}{  \raisebox{0.4ex}{\scriptsize{$\bullet$}} $\M_0 \BD \phi$ }};
            \node[anchor=center] (n2)  at (0.9,0.15) {\textcolor{blue}{ \raisebox{0.4ex}{\scriptsize{$\bullet$}} $\M_0 \BD \psi$ }};
\end{scope}
\end{tikzpicture}
\caption{Two $P$-modules $M$ and $N$ (a), the dimension functions of their birth-death modules (b), and the rank of their M\"obius homology (c).}
\label{fig:classical_bd}
\end{center}
\end{figure}

The following example compares two similar modules over the same poset. While their persistence diagrams are identical, their persistent M\"obius homologies differ. This demonstrates that persistent M\"obius homology is a finer invariant than the persistence diagram.

\begin{ex}
Let $P$ be the poset in Figure~\ref{fig:poset_int} and consider 
Figure~\ref{fig:birth_death} for an example
of two $P$-modules, their free presentations, and their birth-death modules.
Note that the dimension functions of both birth-death modules are the same and therefore their persistence diagrams are the  same.
However, their persistent M\"obius homology does see a difference.
The table below lists the M\"obius homology of the two birth-death modules at the interval~$[a,\infty)$.

    \begin{center}
    \begin{tabular}{L || C | C } 
    [a,\infty) &  \phi & \psi \\  \hline \hline
    \M_0 & \field & 0 \\\hline
    \M_1 & \field & 0 \\ \hline
    \M_2 & 0 & 0
    \end{tabular}
    \end{center}
\end{ex}

A third invariant worth considering is the total homology $H_\ast \big( \Delta(\Int P)); \underline{\BD\phi} \big)$ of the birth-death cosheaf.
By Corollary~\ref{cor:total_homology},
    \begin{equation*} \label{eq:total_homology}
    \sum_{[a,b) \in \Int P} \partial (\BD \phi)[a,b) = \sum_{d  \geq 0} (-1)^d \Big[ H_d \big(\Delta (\Int P); \underline{\BD \phi} \big) \Big].
    \end{equation*}

\begin{figure}
    \centering
    \begin{subfigure}[b]{0.45\textwidth}
    \centering
    \begin{equation*}
    \begin{tikzcd}[ampersand replacement=\&]
    \field \ar[r, "0"] \& 0 \\
    \field^2 \ar[u, "\footnotesize{\begin{pmatrix} 1 & 0  \end{pmatrix}}"] \ar[r, "\footnotesize{\begin{pmatrix} 1 & 0  \end{pmatrix}}"']
    \& \field \ar[u]
    \end{tikzcd}
    \end{equation*}
    \caption{$M : P \to \vec$}
    \label{fig:vidit_ex_1}
    \end{subfigure}  \hfill
    \begin{subfigure}[b]{0.45\textwidth}
    \centering
    \begin{equation*}
    \begin{tikzcd}[ampersand replacement=\&] \field \ar[r, "0"] \& 0 \\
    \field^2 \ar[u, "\footnotesize{\begin{pmatrix} 0 & 1  \end{pmatrix}}"] \ar[r, "\footnotesize{\begin{pmatrix} 1 & 0  \end{pmatrix}}"'] \& \field \ar[u]
    \end{tikzcd}
    \end{equation*}
    \caption{$N : P \to \vec$}
    \label{fig:vidit_ex_2}
    \end{subfigure}
    
    \begin{subfigure}[b]{0.45\textwidth}
    \centering
    \begin{equation*}
    \begin{tikzcd}[ampersand replacement=\&] 
    \field \oplus \field \ar[rrr, "\One"] \ar[dr, twoheadrightarrow, "\footnotesize{\begin{pmatrix} 1 & 0 \end{pmatrix}}"]
    \& \& \& \field \oplus \field \ar[ld, twoheadrightarrow, "\One"] \\
    \& \field \ar[r, "0"] \& 0 \& \\
    \& \field^2 \ar[u, "\footnotesize{\begin{pmatrix} 1 & 0  \end{pmatrix}}"] \ar[r, "\footnotesize{\begin{pmatrix} 1 & 0  \end{pmatrix}}"'] \& \field \ar[u] \& \\
    \field \oplus \field \ar[uuu, "\One"] \ar[rrr, "\One"] \ar[ru, "\One"] \& \& \& \field \oplus \field \ar[ul, twoheadrightarrow, "\footnotesize{\begin{pmatrix} 1 & 0 \end{pmatrix}}"] \ar[uuu, "\One"]
    \end{tikzcd}
    \end{equation*}
    \caption{$\phi : F \Rightarrow M$}
    \label{fig:free_1}
    \end{subfigure}  \hfill
    \begin{subfigure}[b]{0.45\textwidth}
    \centering
    \begin{equation*}
    \begin{tikzcd}[ampersand replacement=\&] 
    \field \oplus \field \ar[rrr, "\One"] \ar[dr, twoheadrightarrow, "\footnotesize{\begin{pmatrix} 0 & 1 \end{pmatrix}}"]
    \& \& \& \field \oplus \field \ar[ld, twoheadrightarrow, "\One"] \\
    \& \field \ar[r, "0"] \& 0 \& \\
    \& \field^2 \ar[u, "\footnotesize{\begin{pmatrix} 0 & 1  \end{pmatrix}}"] \ar[r, "\footnotesize{\begin{pmatrix} 1 & 0  \end{pmatrix}}"'] \& \field \ar[u] \& \\
    \field \oplus \field \ar[uuu, "\One"] \ar[rrr, "\One"] \ar[ru, "\One"] \& \& \& \field \oplus \field \ar[ul, twoheadrightarrow, "\footnotesize{\begin{pmatrix} 1 & 0 \end{pmatrix}}"] \ar[uuu, "\One"]
    \end{tikzcd}
    \end{equation*}
    \caption{$\psi : G \Rightarrow N$}
    \label{fig:free_2}
    \end{subfigure}
     \begin{subfigure}[b]{0.45\textwidth}
    \centering
    \begin{equation*}
    \begin{tikzcd}[ampersand replacement=\&] \& \field^2 \& \\
    \field^2 \ar[ur,"\One"] \& \& \field^2 \ar[swap,ul,"\One"] \\
    \field \ar[u,"\footnotesize{\begin{pmatrix} 0 \\ 1   \end{pmatrix}}"]  \& \field^2 \ar[ru,"\One"] 
    \ar[swap,lu,"\One"] \& \field \ar[swap,u,"\footnotesize{\begin{pmatrix} 0 \\ 1  \end{pmatrix}}"] \\
    \field \ar[u,"\One"] \ar[ur, "\footnotesize{\begin{pmatrix} 0 \\ 1  \end{pmatrix}}"'] \&\& \field \ar[swap,u,"\One"] 
    \ar[ul, "\footnotesize{\begin{pmatrix} 0 \\ 1  \end{pmatrix}}"] \\
    \& 0 \ar[lu,"\One"] \ar[swap,ru,"\One"] \&
    \end{tikzcd}
    \end{equation*}
    \caption{$\BD \phi : \Int P \to \vec$}
    \label{fig:BD_1}
    \end{subfigure}
    % \begin{subfigure}[b]{0.45\textwidth}
    % \centering
    % \begin{equation*}
    % \begin{tikzcd}[ampersand replacement=\&] \& 0 \& \\
    % \field \ar[ur] \& \& \field \ar[ul] \\
    % 0 \ar[u]  \& \field^2 \ar[ru, "\footnotesize{\begin{pmatrix} 0 \\ 1  \end{pmatrix}}"] 
    % \ar[lu, "\footnotesize{\begin{pmatrix} 0 \\ 1  \end{pmatrix}}"'] \& 0 \ar[u] \\
    % \field \ar[u] \ar[ur, "\footnotesize{\begin{pmatrix} 0 \\ 1  \end{pmatrix}}"'] \&\& \field \ar[u] 
    % \ar[ul, "\footnotesize{\begin{pmatrix} 0 \\ 1  \end{pmatrix}}"] \\
    % \& 0 \ar[lu] \ar[ru] \&
    % \end{tikzcd}
    % \end{equation*}
    % \caption{$\BD \phi : \Int P \to \vec$}
    % \label{fig:BD_1}
    % \end{subfigure} 
    \hfill
    % \begin{subfigure}[b]{0.45\textwidth}
    % \centering
    % \begin{equation*}
    % \begin{tikzcd}[ampersand replacement=\&] \& 0 \& \\
    % \field \ar[ur] \& \& \field \ar[ul] \\
    % 0 \ar[u]  \& \field^2 \ar[ru, "\footnotesize{\begin{pmatrix} 1 \\ 0  \end{pmatrix}}"] 
    % \ar[lu, "\footnotesize{\begin{pmatrix} 0 \\ 1  \end{pmatrix}}"'] \& 0 \ar[u] \\
    % \field \ar[u] \ar[ur, "\footnotesize{\begin{pmatrix} 1 \\ 0  \end{pmatrix}}"'] \&\& \field \ar[u] 
    % \ar[ul, "\footnotesize{\begin{pmatrix} 0 \\ 1  \end{pmatrix}}"] \\
    % \& 0 \ar[lu] \ar[ru] \&
    % \end{tikzcd}
    % \end{equation*}
    % \caption{$\BD \psi : \Int P \to \vec$}
    % \label{fig:BD_2}
    % \end{subfigure}
     \begin{subfigure}[b]{0.45\textwidth}
    \centering
    \begin{equation*}
    \begin{tikzcd}[ampersand replacement=\&] \& \field^2 \& \\
    \field^2 \ar[ur,"\One"] \& \& \field^2 \ar[swap,ul,"\One"] \\
    \field \ar[u,"\footnotesize{\begin{pmatrix} 1 \\ 0   \end{pmatrix}}"]  \& \field^2 \ar[ru,"\One"] 
    \ar[swap,lu,"\One"] \& \field \ar[swap,u,"\footnotesize{\begin{pmatrix} 0 \\ 1  \end{pmatrix}}"] \\
    \field \ar[u,"\One"] \ar[ur, "\footnotesize{\begin{pmatrix} 1 \\ 0  \end{pmatrix}}"'] \&\& \field \ar[swap,u,"\One"] 
    \ar[ul, "\footnotesize{\begin{pmatrix} 0 \\ 1  \end{pmatrix}}"] \\
    \& 0 \ar[lu,"\One"] \ar[swap,ru,"\One"] \&
    \end{tikzcd}
    \end{equation*}
    \caption{$\BD \psi : \Int P \to \vec$}
    \label{fig:BD_2}
    \end{subfigure}
    \caption{Two $P$-modules $M$ and $N$, their free presentations $\phi$ and $\psi$, and their birth-death modules
    $\BD \phi$ and $\BD \psi$.}
    \label{fig:birth_death}
\end{figure}

%%%%%%%%%%%%%%%%%%%%%%%%%%%%%%%%%%%%%%%%%%%%%%%%%
\subsection{Independence of Free Presentation}
%%%%%%%%%%%%%%%%%%%%%%%%%%%%%%%%%%%%%%%%%%%%%%%%%
Our definition of persistent M\"obius homology depends on the choice of a free presentation.
In this section, we show that the persistent M\"obius homology at intervals $[a,b)$ off the diagonal, i.e., $a< b$, is independent of this choice.
Traditionally, persistence is not concerned about intervals on the diagonal as they have zero persistence.

To state the independence of the free presentation, we require a canonical intermediate \(P\)-module for a given \(P\)-module \(M\).

\begin{defn}
The \define{kernel} $P$-module of $M$, denoted $K_M : \Int P \to \cat$, is defined as follows. For every interval $[a, b)$, where $b \neq \infty$, $K_M$ assigns the kernel of the morphism $M(a \leq b)$. For intervals of the form $[a, \infty)$, $K_M$ assigns the object $M(a)$. For every $[a, b) \preceq [c, d)$, $K_M$ assigns the appropriate composition of morphisms below, depending on whether $b$ or $d$ is $\infty$:
\[
\begin{tikzcd}
    \ker M(a \leq b) \ar[rr, dashrightarrow] \ar[d, hookrightarrow]
    && \ker M(c \leq d) \ar[d, hookrightarrow] \\
    M(a) \ar[rr, "M(a \leq c)"] && M(c).
\end{tikzcd}
\]
\end{defn}

The following proposition establishes that the persistent M\"obius homology is independent of the choice of a free presentation for all intervals off the diagonal.

\begin{prop} \label{prop:free}
Let \(\phi : F \Rightarrow M\) be a free presentation of a \(P\)-module \(M\). 
For every \(a < b\), we have 
$\M_\ast \BD \phi [a,b) \cong \M_\ast K_M [a,b) .$
\end{prop} 

Before presenting the proof, we revisit the persistence diagrams of the modules in Example~\ref{ex:classical_pd_bd}, this time using the kernel module.

\begin{ex}\label{ex:classical_pd_ker}
Consider the $P$-modules $M, N : P \to \vec$ illustrated in Figure~\ref{fig:classical_ker}. The dimension functions $k_m, k_n : \Int P \to \Z$ corresponding to their kernel modules $K_M$ and $K_N$ are shown. To the right, the dimensions of their persistent M\"obius homologies $\M_d K_M$ and $\M_d K_N$ are displayed, with non-zero values appearing only in two dimensions: $d = 0$ and $d = 1$.

This shows that the Euler characteristics $\partial k_m$ and $\partial k_n$ of $\M_\ast K_M$ and $\M_\ast K_N$, respectively, take negative values along the diagonal. Off the diagonal, however, $\partial k_m$ and $\partial k_n$ match $\partial (\bd \phi)$ and $\partial (\bd \psi)$ from Figure~\ref{fig:classical_bd}. As a result, Proposition~\ref{prop:free} still applies in this case.
\end{ex}

\begin{figure}
\begin{center}
\begin{tikzpicture}
        \node[anchor=south west, inner sep=0] (image) at (0,0) {\includegraphics[width=0.9\textwidth,page=7]{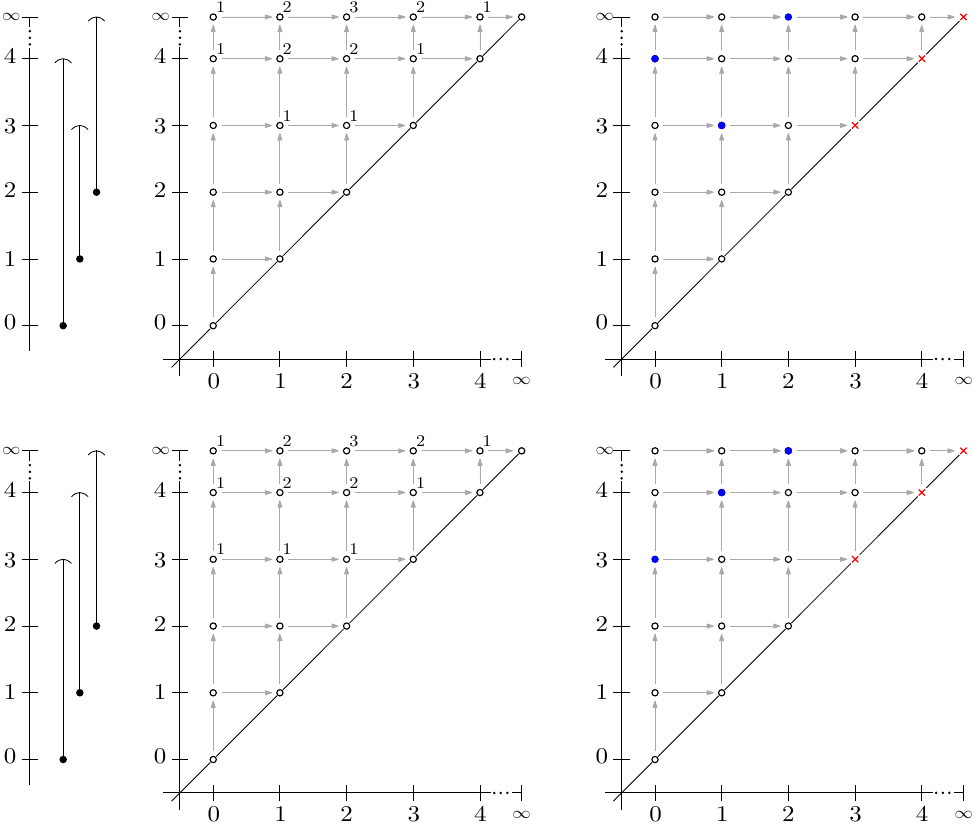}}; %
        \begin{scope}[x={(image.south east)}, y={(image.north west)}] 
            \node[anchor=center] (M1)  at (-0.05,0.78) {$M$};
            \node[anchor=center] (M2)  at (-0.05,0.255) {$N$};
            \node[anchor=center] (a)  at (0.07,-0.04) {(a)};
            \node[anchor=center] (b)  at (0.355,-0.04) {(b)};
            \node[anchor=center] (c)  at (0.81,-0.04) {(c)};
             \node[anchor=center] (n1)  at (0.45,0.65) {$k_m$};
            \node[anchor=center] (n2)  at (0.45,0.15) {$k_n$};
              \node[anchor=center] (n1)  at (0.9,0.71) {\textcolor{blue}{ \raisebox{0.4ex}{\scriptsize{$\bullet$}} $\M_0 K_M$ }};
              \node[anchor=center] (n1a)  at (0.9,0.65) {\textcolor{red}{ 
\raisebox{0.4ex}{\scriptsize{$\times$}} $\M_1 K_M$ }};
                \node[anchor=center] (n2)  at (0.9,0.19) {\textcolor{blue}{\raisebox{0.4ex}{\scriptsize{$\bullet$}} $ \M_0 K_N$ }};
              \node[anchor=center] (n2a)  at (0.9,0.13) {\textcolor{red}{ \raisebox{0.4ex}{\scriptsize{$\times$}} $\M_1 K_N$ }};
             % \node[anchor=center] (n1)  at (0.9,0.65) {$\M_\bullet K_M$};
           % \node[anchor=center] (n2)  at (0.9,0.15) {$\M_\bullet K_N$};
\end{scope}
\end{tikzpicture}
\caption{Two $P$-modules $M$ and $N$ (a), the dimension functions of their kernel modules (b), and the rank of their M\"obius homology (c).}
\label{fig:classical_ker}
\end{center}
\end{figure}

The proof of Proposition~\ref{prop:free} involves an intermediate module. Let \(D_\phi : \Int P \to \cat\) be the module that assigns to every interval \([a,b)\) the birth-death object $\BD \phi  [a, a)$  and to every inclusion \([a,b) \preceq [c,d)\) the canonical monomorphism $\BD \phi [a,a) \hookrightarrow \BD \phi [c,c)$.

\begin{lem} 
\label{lem:zero_homology}
For all $a < b$ in $P$, $\M_\ast D_\phi [a,b) = 0$.
\end{lem}
\begin{proof}
Let $f : P \to \Int P$ be the monotone function $f(a) := [a,a)$ and $g: \Int P \to P$ the monotone function
$g[b,c) := b$.
For every $a \in P$ and $[b,c) \in \Int P$ the following holds:
$$f(a) = [a,a) \leq [b,c) \text{ iff } a \leq g [b,c) = b.$$
This makes 
$f : P \leftrightarrows \Int P : g$ a Galois connection.
Let $B : P \to \cat$ be the module that assigns to every $a \in P$ the object $\BD \phi [a,a)$ and to every $a \leq b$ the canonical monomorphism between the two birth-death objects. 
Note $D_\phi = B \circ g$ and $f^{-1}[a,b)$
is empty unless $a =b$.
The desired result follows from Theorem~\ref{thm:rota_homology}.
\end{proof}

\begin{proof}[Proof of Proposition~\ref{prop:free}]
The $P$-modules $D_\phi$, $\BD \phi$, and $K_M$ fit into a short exact sequence:
\begin{equation*}
    \begin{tikzcd}
        0 \ar[r] & D_\phi \ar[r]  & \BD \phi 
        \ar[r] & K_M \ar[r] & 0.
    \end{tikzcd}
\end{equation*}
This short exact sequence of modules gives rise to a short  exact sequence
of M\"obius chain complexes:
    \begin{equation*}
    \begin{tikzcd}
        0 \ar[r] & \Ch_\bullet D _\phi [a,b) \ar[r]  & \Ch_\bullet (\BD \phi) [a,b) \ar[r] & \Ch_\bullet K_M [a,b) \ar[r] & 0.
    \end{tikzcd}
    \end{equation*}
This short exact sequence of M\"obius chain complexes gives rise to  a long exact sequence of M\"obius homology objects:
    \begin{equation*}
    \begin{tikzcd}
        \cdots \ar[r] & \M_{n+1} K_M [a,b)  \ar[r] & \M_n D_\phi [a,b) \ar[r] & \M_n (\BD \phi) [a,b) \ar[d, "i"] \\
        \cdots & \M_{n-1} (\BD \phi) [a,b) \ar[l] & \M_{n-1} D_\phi [a,b) \ar[l]
        & \M_{n} K_M [a,b) \ar[l]
    \end{tikzcd}
    \end{equation*}
By Lemma~\ref{lem:zero_homology},
both $\M_n D_\phi [a,b)$ and
$\M_{n-1} D_\phi [a,b)$ are zero.
This makes $i$ is an isomorphism.
\end{proof}

%%%%%%%%%%%%%%%%%%%%%%%%%%%%%%%%%%%%%%%%%%%%%%
\subsection{Galois Connections and Persistence}
%%%%%%%%%%%%%%%%%%%%%%%%%%%%%%%%%%%%%%%%%%%%%%
The homological version of Rota's Galois connection theorem applies to persistent M\"obius homology.
This follows immediately from the observation that a Galois connection between posets induces a Galois connection between their intervals posets.

A monotone function $f :P \to Q$ induces a monotone function $\Int f : \Int P \to \Int Q$ defined as $\Int f [a,b)  = \big[ f(a), f(b) \big]$.
Similarly, a monotone function $g : Q \to P$ induces a monotone function $\Int g  : \Int Q \to \Int P$.

\begin{lem}
If $f : P \leftrightarrows Q : g$ is a Galois connection, then $\Int f : \Int P \leftrightarrows \Int Q : \Int g$
is a Galois connection.
\end{lem}
\begin{proof}
The following statement is to be verified: for all $[a,b) \in \Int P$ and $[x,y) \in  \Int Q$, $\Int f  [a,b)  \leq [x,y)$ iff $[a,b) \leq \Int g [x,y)$.
Unwind the definition of $\preceq$ to get
$$f(a) \leq x \text{ and } f(b) \leq y \Leftrightarrow  a \leq g(x) \text{ and } b \leq g(y).$$
In other words, $f$ and $g$ must form a Galois connection which is true by assumption.
\end{proof}

Consider a $P$-module $M: P \to \cat$ and
a free presentation $\phi : F \Rightarrow M$.
Given a Galois connection $f : P \leftrightarrows Q: g$, let $N := M \circ g  : Q \to \cat$.
The free presentation $\phi$ restricts to a free presentation
$\psi : G \Rightarrow N$ where
$G := F \circ g$ and $\psi(x) := (\phi \circ g) (x)$.
The following corollary follows from Theorem~\ref{thm:rota_homology}.

\begin{cor}\label{cor:persistent_galois}
Let $f : P \leftrightarrows Q: g$ be a Galois connection.
For a module $M : P \to \cat$ and a free presentation $\phi : F \Rightarrow M$, let $N$ be the $Q$-module $M \circ g$ and $\psi : F \circ g \Rightarrow N$ the free presentation $\phi \circ g$.
For every interval $[x,y) \in \Int Q$,
$$\M_\ast (\BD \psi) [x,y) \cong H_\ast 
\Big( \Delta (\Int f)^{-1}_{\leq [x,y)}, \Delta (\Int f)^{-1}_{< [x,y)}; \BD \phi \Big).$$
 \end{cor}

%%%%%%%%%%%%%%%%%%%%%%%%%%%%%%%%%%%%%%%%%%%%%%
\subsection{Bounding Persistent Homological Dimension}
%%%%%%%%%%%%%%%%%%%%%%%%%%%%%%%%%%%%%%%%%%%%%%
When underlying poset of a $P$-module $M$ is a finite distributive lattice, 
there is a bound on the highest non-trivial dimension of its persistent M\"obius homology.
Since~$P$ is distributive, $\Int P$ is also distributive.
If $a \in P$ covers $m$ elements and $b \in P$ covers $n$ elements, then the interval $[a,b) \in \Int P$ covers $m + n$ elements.
The following corollary follows immediately from Corollary~\ref{cor:depth}.

\begin{cor}
Let $P$ be a distributive lattice and suppose that for $a \leq b$, 
$a$ covers~$m$ elements and~$b$ covers~$n$ elements.
Then for any free presentation $\phi : F \Rightarrow M$ of a $P$-module $M$, $\M_{> m + n} (\BD \phi) [a,b) = 0$.
\end{cor}

We believe the bound of $m + n$ can be improved.
For example, when $P$ is totally ordered, as in the case of classical persistence, $m + n  \leq 2$ but $\M_{> 0} (\BD \phi) = 0$ as follows.

\begin{prop} \label{prop:classical_persistence}
Let $P$ be a totally ordered poset, and let $M$ be a $P$-module. 
For all intervals $[b, d) \in \Int P$, we have $\M_{> 0} (\BD \phi) [b, d) = 0$.
\end{prop}

\begin{proof}
Since $P$ is totally ordered, it is a distributive lattice, and thus $\Int P$ is also a distributive lattice. For $[b, d) \in \Int P$, let $\mathcal{M}\big(C_{[b, d)}\big)$ be its meet-generated sublattice:
\[
\begin{tikzcd}
{[b', d)} \ar[r] & {[b, d)} \\
{[b', d')} \ar[u] \ar[r] & {[b, d')}. \ar[u]
\end{tikzcd}
\]
Here, $[b, d)$ is covered by at most two intervals, $[b', d)$ and $[b, d')$, where $b' \lessdot b$ and $d' \lessdot d$. 
Let $f : \Int P_{\leq [b, d)} \leftrightarrows \mathcal{M}\big(C_{[b, d)}\big) : g$ denote the canonical Galois connection.

Choose a free presentation $\phi : F \Rightarrow M$ and let $J : \mathcal{M}\big(C_{[b, d)}\big) \to \cat$ be the restriction of $\BD \phi$:
\[
\begin{tikzcd}
\BD \phi [b', d) \ar[r, hookrightarrow] & \BD \phi [b, d) \\
\BD \phi [b', d') \ar[u, hookrightarrow] \ar[r, hookrightarrow] & \BD \phi [b, d'). \ar[u, hookrightarrow]
\end{tikzcd}
\]
By Proposition~\ref{prop:pullback}, $J$ is a pullback diagram. Moreover, $J = M \circ g$ and $f^{-1}[b, d) = [b, d)$. By Theorem~\ref{thm:rota_homology}, $\M_\ast (\BD \phi) [b, d) \cong \M_\ast J [b, d)$.

To show $\M_{> 0} J [b, d) = 0$, we construct two intermediate modules $I, K : \mathcal{M}\big(C_{[b, d)}\big) \to \cat$ that fit into a short exact sequence $0 \to I \to J \to K \to 0$.

Define $I$ as the constant $\mathcal{M}\big(C_{[b, d)}\big)$-module:
\[
\begin{tikzcd}
\BD \phi [b', d') \ar[r, "\One"] & \BD \phi [b', d') \\
\BD \phi [b', d') \ar[u, "\One"] \ar[r, "\One"] & \BD \phi [b', d'). \ar[u, "\One"]
\end{tikzcd}
\]
The M\"obius chain complex of $I$ at $[b, d)$ is:
\bgroup
\renewcommand{\arraystretch}{0.7}
\setlength{\arraycolsep}{2pt}
\[
\begin{tikzcd}[ampersand replacement=\&, column sep=2em,/tikz/column 3/.style={column sep=4em},/tikz/column 4/.style={column sep=4em}]
\cdots \ar[r] \& 0 \ar[r] \& \begin{matrix}\BD \phi [b', d')\\ \oplus\\\BD \phi [b', d')\end{matrix} \ar[r, "\footnotesize{\begin{pmatrix} \One & 0 \\-\One&-\One\\0 &\One\\\end{pmatrix}}"] \& 
\begin{matrix}\BD \phi [b', d')\\ \oplus\\\BD \phi [b', d')\\ \oplus\\\BD \phi [b', d')\end{matrix} 
 \ar[r, "\footnotesize{\begin{pmatrix} \One & \One &\One \end{pmatrix}}"] \& \BD \phi [b', d') \ar[r] \& 0
\end{tikzcd}
\]
\egroup
implying $\M_\ast I [b, d) = 0$.

Let $K := J / I$ be the quotient module:
\[
\begin{tikzcd}
\dfrac{\BD \phi [b', d)}{\BD \phi [b', d')} \ar[r, hookrightarrow] & \dfrac{\BD \phi [b, d)}{\BD \phi [b', d')} \\
0 = \dfrac{\BD \phi [b', d')}{\BD \phi [b', d')} \ar[u, hookrightarrow] \ar[r, hookrightarrow] & \dfrac{\BD \phi [b, d')}{\BD \phi [b', d')}. \ar[u, hookrightarrow]
\end{tikzcd}
\]
Since $K$ is the quotient of a pullback diagram by a pullback subdiagram, it is also a pullback diagram. The chain complex for $\M_n K[b, d)$ has $\partial_1$ as the direct sum of two monomorphisms:
\[
\begin{tikzcd}
\cdots \ar[r] & 0 \ar[r] & \dfrac{\BD \phi [b', d)}{\BD \phi [b', d')} \oplus \dfrac{\BD \phi [b, d')}{\BD \phi [b', d')} \ar[r, "\partial_1"] & \dfrac{\BD \phi [b, d)}{\BD \phi [b', d')} \ar[r] & 0.
\end{tikzcd}
\]
The morphism $\partial_1$ must be a monomorphism because $K$ is a pullback diagram with the pullback being $0$.
Thus $\M_{> 0} K([b,d)) = 0$.

The short exact sequence $0 \to I \to J \to K \to 0$ induces a long exact sequence in M\"obius homology:
\[
\begin{tikzcd}
\cdots \ar[r] & \M_{n+1} I [b, d) \ar[r] & \M_n J [b, d) \ar[r, "i"] & \M_n K [b, d) \ar[d] \\
\cdots & \M_{n-1} K [b, d) \ar[l] & \M_{n-1} J [b, d) \ar[l] & \M_{n} I [b, d). \ar[l]
\end{tikzcd}
\]
Since $\M_\ast I [b, d) = 0$, $i$ is an isomorphism. Therefore, 
$$\M_{> 0} (\BD \phi) [b, d) \cong \M_{> 0} J [b, d) \cong \M_{> 0} K [b, d) = 0.$$
\end{proof}

%%%%%%%%%%%%%%%%%%%%%%%%%%%%%%%%%%%%%%%%%%%%%%
%%%%%%%%%%%%%%%%%%%%%%%%%%%%%%%%%%%%%%%%%%%%%%
\section{Statements and Declarations}
%%%%%%%%%%%%%%%%%%%%%%%%%%%%%%%%%%%%%%%%%%%%%%
%%%%%%%%%%%%%%%%%%%%%%%%%%%%%%%%%%%%%%%%%%%%%%
\paragraph{Conflict of Interest}
The authors have no financial or proprietary interests in any material discussed in this article.

%%%%%%%%%%%%%%%%%%%%%%%%%%%%%%%%%%%%%%%%%%%%%%
\bibliographystyle{plain}
\bibliography{references}
%%%%%%%%%%%%%%%%%%%%%%%%%%%%%%%%%%%%%%%%%%%%%%

%%%%%%%%%%%%%%%%%%%%%%%%%%%%%%%%%%%%%%%%%%%%%%
%%%%%%%%%%%%%%%%%%%%%%%%%%%%%%%%%%%%%%%%%%%%%%
\appendix
%%%%%%%%%%%%%%%%%%%%%%%%%%%%%%%%%%%%%%%%%%%%%%
%%%%%%%%%%%%%%%%%%%%%%%%%%%%%%%%%%%%%%%%%%%%%%

%%%%%%%%%%%%%%%%%%%%%%%%%%%%%%%%%%%%%%%%%%%%%%%
\section{Proof of Proposition~\ref{prop:chain_eq}}
\label{sec:appendix_one}
%%%%%%%%%%%%%%%%%%%%%%%%%%%%%%%%%%%%%%%%%%%%%%%

This section proves Proposition~\ref{prop:chain_eq} in detail through two steps. First, Lemma~\ref{lem:left_chain_homotopy} shows that $\Delta g \circ \Delta f$ is chain homotopic to $\One_{C_\bullet(\Delta P; \Z)}$. Then, Lemma~\ref{lem:right_chain_homotopy} shows that $\Delta f \circ \Delta g$ is chain homotopic to $\One_{C_\bullet(\Delta Q; \Z)}$.

%%%%%%%%%%%%%%%%%%%%%%%%%%%%%%%%%%%%%%%%%%%%%%%
\paragraph{Left Chain Homotopy}
%%%%%%%%%%%%%%%%%%%%%%%%%%%%%%%%%%%%%%%%%%%%%%%

Consider an $n$-simplex in $\Delta P$ given as a chain $a_0  < \cdots < a_n$.
Let $b_i = g \circ f(a_i)$.
The following relations follow from Lemma~\ref{lem:galois}~(i) combined with the fact that
$g \circ f$ is monotone:
	\begin{equation*}
	\begin{tikzcd}
	b_0 \ar[r, "\leq"] & b_1 \ar[r, "\leq"] & \ar[r, "\leq"]  \cdots \ar[r, "\leq"] &  b_{n-1} \ar[r, "\leq"] & b_n \\
	a_0 \ar[r, "<"] \ar[u, "\leq"] & a_1 \ar[r, "<"] \ar[u, "\leq"] & \ar[r, "<"]  \cdots \ar[r, "<"]  
	&  a_{n-1} \ar[r, "<"] \ar[u, "\leq"] & a_n \ar[u, "\leq"].
	\end{tikzcd}
	\end{equation*}
The sequence $b_0 \leq \cdots \leq b_n$ may contain repetitions in which case it is not an $n$-simplex. 
Consider the linear map $\phi_n : C_n (\Delta P; \Z) \to C_{n+1}(\Delta P; \Z)$ generated by the following simplex-wise
assignment:
	\begin{equation}\label{eq:left_chain_homotopy}
	\phi_n (a_0 < \cdots < a_n) := \sum_{i=0}^n (-1)^i a_0 < \cdots < a_i \leq b_i \leq \cdots \leq b_n.
	\end{equation}
Identify every term on the right-hand side that is not an $(n+1)$-simplex to $0$.

\begin{lem}\label{lem:left_chain_homotopy}
The composition $\Delta g \circ \Delta f : C_\bullet (\Delta P; \Z) \to C_\bullet (\Delta P; \Z)$ is chain homotopic to the identity.
\end{lem}
\begin{proof}
For any $n$-simplex $a_0 < \cdots < a_n$ in $\Delta P$ and $b_i = g \circ f(a_i)$ as above, 
we show 
$$(b_0 \leq \cdots \leq b_n) - (a_0 < \cdots < a_n) = \partial \circ \phi (a_0 < \cdots < a_n) + 
\phi \circ \partial (a_0 < \cdots < a_n).$$
Expand the first term to
	\begin{align*}
	\partial \circ \phi (a_0 < \cdots < a_n) =& 
		\; \partial \left(  \sum_{i=0}^n (-1)^i a_0 < \cdots < a_i \leq b_i \leq \cdots \leq b_n \right) \\
		=& \; \sum_{i=0}^n (-1)^i \Bigg( \sum_{j=0}^i (-1)^j a_0 < \cdots < \widehat{a}_j < \cdots 
		< a_i \leq b_i \leq \cdots \leq b_n \\
		&+  \sum_{j=i}^n (-1)^{j+1} a_0 < \cdots < a_i \leq b_i \leq \cdots \leq \widehat{b}_j \leq \cdots \leq b_n \Bigg).
	\end{align*}
Expand the second term to 
	\begin{align*}
	\phi \circ \partial (a_0 < \cdots < a_n) =&\; \phi \Bigg( \sum_{j=0}^n (-1)^j a_0 < \cdots < \widehat{a}_j < \cdots < a_n \Bigg) \\
	=& \; \sum_{j=0}^n (-1)^j 
	\Bigg( \sum_{i=0}^{j-1} (-1)^i a_0 < \cdots < a_i \leq b_i \leq \cdots  \leq \widehat{a}_j \leq \cdots \leq b_n  \\
	&+ \sum_{i = j+1}^n (-1)^{i-1} a_0 < \cdots < \widehat{a}_j < \cdots < a_i \leq b_i \leq \cdots \leq b_n \Bigg).
	\end{align*}
Note that every term in the second expansion has a corresponding term in the first expansion but with an opposite sign.
This leaves just $(b_0 \leq \cdots \leq b_n) - (a_0 < \cdots < a_n)$ as desired.
\end{proof}

%%%%%%%%%%%%%%%%%%%%%%%%%%%%%%%%%%%%%%%%%%%%%%%
\paragraph{Right Chain Homotopy}
%%%%%%%%%%%%%%%%%%%%%%%%%%%%%%%%%%%%%%%%%%%%%%%
Consider an $n$-simplex in $\Delta Q$ given as a chain $x_0  < \cdots < x_n$.
Let $y_i = f \circ g(x_i)$.
The following relations follow from Lemma~\ref{lem:galois}~(ii) combined with the fact that
$f \circ g$ is monotone:
	\begin{equation*}
	\begin{tikzcd}
	y_0 \ar[r, "\leq"] \ar[d, "\leq"]  & y_1 \ar[r, "\leq"] \ar[d, "\leq"] & \ar[r, "\leq"]  \cdots \ar[r, "\leq"] &  
	y_{n-1} \ar[r, "\leq"] \ar[d, "\leq"] & y_n \ar[d, "\leq"] \\
	x_0 \ar[r, "<"]  & x_1 \ar[r, "<"] & \ar[r, "<"]  \cdots \ar[r, "<"]  
	&  x_{n-1} \ar[r, "<"]  & x_n .
	\end{tikzcd}
	\end{equation*}
The sequence $y_0 \leq \cdots \leq y_n$ may contain repetitions in which case it is not an $n$-simplex. 
Consider the linear map $\psi_n : C_n (\Delta Q; \Z) \to C_{n+1}(\Delta Q; \Z)$ generated by the following simplex-wise
assignment:
	\begin{equation}\label{eq:right_chain_homotopy}
	\psi_n (x_0 < \cdots < x_n) := \sum_{i=0}^n (-1)^i y_0 \leq \cdots \leq y_i \leq x_i < \cdots < x_n.
	\end{equation}
Identify every term on the right-hand side that is not an $(n+1)$-simplex to $0$.

\begin{lem}\label{lem:right_chain_homotopy}
The composition $\Delta f \circ \Delta g : C_\bullet (\Delta Q; \Z) \to C_\bullet (\Delta Q; \Z)$ is chain homotopic to the identity.
\end{lem}
\begin{proof}
For any $n$-simplex $x_0 < \cdots < x_n$ in $\Delta Q$ and $y_i = f \circ g(x_i)$ as above, 
we show 
$$(x_0 < \cdots < x_n) - (y_0 \leq \cdots \leq y_n) = \partial \circ \psi (x_0 < \cdots < x_n) + \psi \circ \partial (x_0 < \cdots < x_n).$$
Expand the first term to
	\begin{align*}
	\partial \circ \psi (x_0 < \cdots < x_n) =& 
		\; \partial \left(  \sum_{i=0}^n (-1)^i y_0 \leq \cdots \leq y_i \leq x_i < \cdots < x_n \right) \\
		=& \;  \sum_{i=0}^n (-1)^i \Bigg( \sum_{j=0}^i (-1)^j y_0 \leq \cdots \leq \widehat{y}_j \leq \cdots \leq y_i \leq x_i 
		< \cdots < x_n \\
		&+  \sum_{j=i}^n (-1)^{j+1} y_0 \leq \cdots \leq y_i \leq x_i < \cdots < \widehat{x}_j < \cdots < x_n \Bigg).
	\end{align*}
Expand the second term to 
	\begin{align*}
	\psi \circ \partial (x_0 < \cdots < x_n) =&\; \psi \Bigg( \sum_{j=0}^n (-1)^j x_0 < \cdots < \widehat{x}_j < \cdots < x_n \Bigg) \\
	=& \; \sum_{j=0}^n (-1)^j 
	\Bigg( \sum_{i=0}^{j-1} (-1)^i y_0 \leq \cdots \leq y_i \leq x_i < \cdots < \widehat{x}_j < \cdots < x_n  \\
	&+ \sum_{i = j+1}^n (-1)^{i-1} y_0 \leq \cdots \leq \widehat{y}_j \leq \cdots \leq y_i \leq x_i < \cdots < x_n \Bigg).
	\end{align*}
Note that every term in the second expansion has a corresponding term in the first expansion but with an opposite sign.
This leaves just $(x_0 < \cdots < x_n) - (y_0 \leq \cdots \leq y_n)$ as desired.
\end{proof}

\end{document}